\documentclass[10pt]{amsart}

\usepackage{amsmath, amssymb, amsthm}
\usepackage{enumerate}
\usepackage{graphicx}
\usepackage{hyperref}
\usepackage{makecell}
\usepackage{tikz}
\usepackage{array}
\usepackage{longtable}
\usepackage{afterpage}

\usetikzlibrary{positioning}
\usetikzlibrary{arrows,automata}

\tikzset{
    position/.style args={#1:#2 from #3}{
        at=(#3.#1), anchor=#1+180, shift=(#1:#2)
    }
}

\newtheorem{thm}{Theorem}[section]
\newtheorem{lemma}[thm]{Lemma}
\newtheorem{prop}[thm]{Proposition}
\newtheorem{cor}[thm]{Corollary}

\newtheorem*{thm:thm2}{Theorem \ref{thm2}}

\theoremstyle{definition}

\newtheorem{example}[thm]{Example}

\newtheorem{question}[thm]{Question}

\theoremstyle{remark}

\newcommand{\Id}{\mathop{\rm Id}}

\begin{document}

\author[S. Canilang]{Sara Canilang}
\email[1]{sara.canilang@gmail.com}

\author[M. P. Cohen]{Michael P. Cohen}
\email[2]{mcohen@carleton.edu}

\author[N. Graese]{Nicolas Graese}
\email[3]{graesenah@gmail.com}

\author[I. Seong]{Ian Seong}
\email[4]{ianiceman7166@gmail.com}

\address{Department of Mathematics and Statistics,
Carleton College,
One North College Street,
Northfield, MN 55057}



\title{The closure-complement-frontier problem in saturated polytopological spaces}

\begin{abstract}  Let $X$ be a space equipped with $n$ topologies $\tau_1,...,\tau_n$ which are pairwise comparable and saturated, and for each $1\leq i\leq n$ let $k_i$ and $f_i$ be the associated topological closure and frontier operators, respectively.  Inspired by the closure-complement theorem of Kuratowski, we prove that the monoid of set operators $\mathcal{KF}_n$ generated by $\{k_i,f_i:1\leq i\leq n\}\cup\{c\}$ (where $c$ denotes the set complement operator) has cardinality no more than $2p(n)$ where $p(n)=\frac{5}{24}n^4+\frac{37}{12}n^3+\frac{79}{24}n^2+\frac{101}{12}n+2$.  The bound is sharp in the following sense: for each $n$ there exists a saturated polytopological space $(X,\tau_1,...,\tau_n)$ and a subset $A\subseteq X$ such that repeated application of the operators $k_i, f_i, c$ to $A$ will yield exactly $2p(n)$ distinct sets.  In particular, following the tradition for Kuratowski-type problems, we exhibit an explicit initial set in $\mathbb{R}$, equipped with the usual and Sorgenfrey topologies, which yields $2p(2)=120$ distinct sets under the action of the monoid $\mathcal{KF}_2$.
\end{abstract}

\maketitle

\section{Introduction}

In his 1922 thesis \cite{Kuratowski_1922a}, Kuratowski posed and solved the following problem: given a topological space $(X,\tau)$, what is the largest number of distinct subsets that can be obtained by starting from an initial set $A\subseteq X$, and applying the topological closure and complement operators, in any order, as often as desired?  The answer is $14$.  This result, now widely known as Kuratowski's \textit{closure-complement theorem}, is both thought-provoking and amusing, and has inspired a substantial number of authors to study generalizations, variants, and elaborations of the original closure-complement problem.  We recommend consulting the admirable survey of Gardner and Jackson \cite{Gardner_Jackson_2008a}, or visiting Bowron's website \textit{Kuratowski's Closure-Complement Cornucopia} \cite{Bowron_2012a} for an indexed list of all relevant literature.

Shallit and Willard \cite{Shallit_Willard_2011a} considered a natural extension of Kuratowski's problem.  If we equip a space $X$ with not one but two distinct topologies $\tau_1$ and $\tau_2$, how many distinct subsets may be obtained by starting with an initial set, and applying each of the two associated closure operators $k_1$, $k_2$, and the set complement operator $c$, in any order, as often as desired?  The authors construct an example of a bitopological space $(X,\tau_1,\tau_2)$ where it is possible to obtain infinitely many subsets from a certain initial set.  Consequently, the monoid $\mathcal{K}_2$ of set operators generated by $\{k_1,k_2,c\}$ may have infinitely many elements in general.  In their example, the topologies $\tau_1$ and $\tau_2$ are incomparable, which suggests that the monoid may yet be finite in case $\tau_1\supseteq\tau_2$.

In \cite{Banakh_2018a}, Banakh, Chervak, Martynyuk, Pylypovych, Ravsky, and Simkiv verify this last possibility, and generalize the closure-complement theorem to polytopological spaces, i.e. sets $X$ equipped with families of topologies $\mathcal{T}$ in which the topologies are linearly ordered by inclusion.  If the family is a finite set $\mathcal{T}=\{\tau_1,...,\tau_n\}$, they give an explicit formula for the maximal cardinality of the monoid $\mathcal{K}_n$ generated by $\{k_j:1\leq j\leq n\}\cup\{c\}$.  This maximal cardinality is of course $14$ when $n=1$, and grows exponentially as $n\rightarrow\infty$.

The authors of \cite{Banakh_2018a} also consider the special case where the topologies involved are \textit{saturated}, i.e., for any $1\leq j,\ell\leq n$, if a nonempty set $U$ is $\tau_j$-open, then $U$ has nonempty $\tau_\ell$-interior.  In the saturated case, the cardinality bound on the monoid is given by $\#\mathcal{K}_n\leq 12n+2$.  The most natural example is the case of the real line $\mathbb{R}$ equipped with $\tau_2=$ the usual topology and $\tau_1=$ the Sorgenfrey topology.  Then one may obtain no more than $12\cdot 2+2=26$ distinct sets by applying $k_1,k_2,c$ to any particular initial set, and indeed this upper bound is obtainable in $(\mathbb{R},\tau_1,\tau_2)$, as demonstrated explicitly in \cite{Banakh_2018a}.

In \cite{Gaida_Eremenko_1974a}, Gaida and Eremenko solved a \textit{closure-complement-frontier problem} by showing that in any topological space $(X,\tau)$, the monoid $\mathcal{KF}_1$ generated by $\{k,f,c\}$ (where $f$ is the frontier operator, or topological boundary operator) has cardinality $\leq 34$; moreover there are examples of spaces in which it is possible to obtain $34$ distinct subsets by applying the operators to a single initial set.  This problem also appeared as Problem E3144 in \textit{American Mathematical Monthly} \cite{Buchman_1986a}.  The purpose of this paper is to study the extension of Gaida and Eremenko's problem to the setting of saturated polytopological spaces as in \cite{Banakh_2018a}.

To state our result, we consider a polytopological space $(X,\tau_1,...,\tau_n)$, and we denote by $\mathcal{KF}_n=\mathcal{KF}_n(X,\tau_1,...,\tau_n)$ the monoid of set operators generated by $\{k_j,f_j:1\leq j\leq n\}\cup \{c\}$.  We also let $\mathcal{KF}_n^0=\mathcal{KF}_n^0(X,\tau_1,...,\tau_n)$ denote the monoid generated by $\{k_j,i_j,f_j:1\leq j\leq n\}$, where $i_j$ is the interior operator associated to $\tau_j$.  Since $i_j=ck_jc$, we have that $\mathcal{KF}_n^0\subseteq\mathcal{KF}_n$, and in fact, in Section 2 we observe that\\

\begin{center} $\mathcal{KF}_n=\mathcal{KF}_n^0\cup c\mathcal{KF}_n^0$
\end{center}
\vspace{.3cm}

\noindent so that $\mathcal{KF}_n^0$ comprises the submonoid of \textit{even operators} of $\mathcal{KF}_n$, and\\

\begin{center} $\#\mathcal{KF}_n=2\cdot\#\mathcal{KF}_n^0$.
\end{center}
\vspace{.3cm}

\noindent Our main theorem follows.

\begin{thm} \label{thm_main}  Let $(X,\tau_1,...,\tau_n)$ be a saturated polytopological space.  Then $\#\mathcal{KF}_n^0\leq p(n)$ and $\#\mathcal{KF}_n=2\cdot\#\mathcal{KF}_n^0\leq 2p(n)$, where

\begin{center}  $p(n) = \frac{5}{24}n^4+\frac{37}{12}n^3+\frac{79}{24}n^2+\frac{101}{12}n+2$.
\end{center}
\end{thm}

Thus for $n=1$ we recover Gaida-Eremenko's result with $p(n)=17$ and $2p(n)=34$.  The next few upper bounds are $p(2)=60$, $p(3)=157$, $p(4)=339$, and $p(5)=642$.

We also demonstrate that the bound $p(n)$ is sharp.

\begin{thm} \label{thm2}  For every $n\geq 1$, there exists a saturated polytopological space $(X,\tau_1,...,\tau_n)$ in which $\#\mathcal{KF}_n^0=p(n)$ and $\#\mathcal{KF}_n=2p(n)$.  In fact, there is an initial set $A\subseteq X$ such that $\#\{oA:o\in\mathcal{KF}_n\}=2p(n)$.
\end{thm}

The explicit examples we give are natural and easy to understand (disjoint unions of copies of $\mathbb{R}$ equipped with combinations of the Sorgenfrey and Euclidean topologies), but not finite.  By the results of \cite{McKinsey_Tarski_1944a} (see \cite{Gardner_Jackson_2008a} Theorem 4.1 and surrounding remarks), we deduce abstractly that there must exist a finite polytopological space $(X,\tau_1,...,\tau_n)$ on which $\#\mathcal{KF}_n^0=p(n)$, but we do not know how many points are necessary.

\begin{question} \label{question_finite}  What is the minimal cardinality of a polytopological space $(X,\tau_1,...,\tau_n)$ for which $\#\mathcal{KF}_n^0=p(n)$ exactly?  What is the minimal cardinality of a space in which one can find an initial set $A$ with $\#\{oA:o\in\mathcal{KF}_n\}=2p(n)$?
\end{question}

It would be interesting to know the answer even for $n=2$.  It is known that the minimal number of points needed for a space to contain a Kuratowski $14$-set is $7$; see \cite{Herda_Metzler_1966a}.  During the preparation of this article, Bowron has communicated to us that if $n=1$, then the minimal number of points needed for $\#\mathcal{KF}_1^0=17$ is four, while the minimal number of points needed to contain a $34$-set is $8$. 

Another interesting question that remains open is to solve the closure-complement-frontier problem for polytopological spaces which are not necessarily saturated.

\begin{question}  Let $(X,\tau_1,...,\tau_n)$ be a polytopological space which is not necessarily saturated.  What is the maximal cardinality of the monoid $\mathcal{KF}_n$ generated by $\{k_j,f_j:1\leq j\leq n\}\cup\{c\}$?
\end{question}

Finally, it would be interesting to study some of the variants described in Section 4 of \cite{Gardner_Jackson_2008a} in the larger context of polytopological spaces.  For example, it was shown independently by Gardner and Jackson \cite{Gardner_Jackson_2008a} and by Sherman \cite{Sherman_2010a} that in any topological space $(X,\tau)$, the greatest number of sets one may obtain from an initial set $A\subseteq X$ by applying the set operators $\{k,i,\cup,\cap\}$ is $35$.

\begin{question}  Let $(X,\tau_1,...,\tau_n)$ be a (saturated?) polytopological space.  What is the largest number of sets one may obtain from an initial set $A\subseteq X$ by applying the set operators $k_j$, $i_j$ ($1\leq j\leq n$), $\cup$, and $\cap$ in any order, as often as desired?
\end{question}

\section{Preliminaries and Notation}

Recall from the introduction that a \textit{polytopological space} is a set $X$ equipped with a family of topologies $\mathcal{T}$ which is linearly ordered by the inclusion relation.  In this paper we will work only with finite families $\mathcal{T}=\{\tau_1,...,\tau_n\}$ and assume $\tau_1\supseteq...\supseteq\tau_n$.  In this case we refer to $(X,\tau_1,...,\tau_n)$ as an \textit{$n$-topological space}.

For each topology $\tau_j$, we permanently associate the \textit{closure operator} $k_j$, the \textit{interior operator} $i_j$, and the \textit{frontier operator} $f_j$.  We use $c$ to denote the \textit{set complement operator}.  The operators $k_j$ and $i_j$ are idempotent, so $k_jk_j=k_j$ and $i_ji_j=i_j$, and the operator $c$ is an involution, so $cc=\Id$, where $\Id$ denotes the \textit{identity operator}.  For each set $A\subseteq X$ we have $f_jA=k_jA\cap k_jcA$; we summarize this symbolically by writing\\

\begin{center} $f_j=k_j\wedge k_jc=k_j\wedge ci_j$.
\end{center}
\vspace{.3cm}

From the identity above, we see that 

\begin{center} \fbox{$f_jc=f_j$.}
\end{center}
\vspace{.3cm}

We permanently denote by $\mathcal{KF}_n=\mathcal{KF}_n(X,\tau_1,...,\tau_n)$ the smallest monoid of set operators which contains $k_j$, $f_j$ ($1\leq j\leq n$) and $c$.  We also denote by $\mathcal{KF}_n^0=\mathcal{KF}_n^0(X,\tau_1,...,\tau_n)$ the smallest monoid of set operators which contains $k_j$, $i_j$, and $f_j$ ($1\leq j\leq n$).  By DeMorgan's laws, we have $ck_jc=i_j$ and thus it is immediate that $\mathcal{KF}_n^0\subseteq\mathcal{KF}_n$.

Since we are requiring that $\mathcal{KF}_n^0$ be a monoid, it contains the identity operator $\Id$.  It also contains the \textit{zero operator} $0$, i.e. the set operator for which $0A=\emptyset$, for every $A\subseteq X$.  This follows from the work of Gaida and Eremenko \cite{Gaida_Eremenko_1974a}, who observed that\\

\begin{center} $i_1f_1k_1=0$.
\end{center}
\vspace{.3cm}

We also define the \textit{one operator} by the rule $1=c0$, so $1A=X$ for every $A\subseteq X$ and $1\in\mathcal{KF}_n$.

\begin{prop}  The sets $\mathcal{KF}_n^0$ and $c\mathcal{KF}_n^0$ are disjoint and $\mathcal{KF}_n$ is equal to their union.
\end{prop}

\begin{proof}  By examining the generators $k_j,i_j,f_j$ of $\mathcal{KF}_n^0$, it is clear that $o\emptyset=\emptyset$ and $co\emptyset=X$ for any operator $o\in\mathcal{KF}_n^0$.  Therefore, $\mathcal{KF}_n^0$ and $c\mathcal{KF}_n^0$ are disjoint.

To see that $\mathcal{KF}_n\subseteq \mathcal{KF}_n^0\cup c\mathcal{KF}_n^0$, we can argue by induction on word length of elements of $\mathcal{KF}_n$.  Let $\mathcal{W}_m\subseteq \mathcal{KF}_n$ be the set of operators which can be written as a word of length $\leq m$ in the generators $k_j,f_j,c$.  Assume that $\mathcal{W}_m\subseteq \mathcal{KF}_n^0\cup c\mathcal{KF}_n^0$ (which is certainly true if $m=1$).  Then $\mathcal{W}_{m+1}$ is the union of sets of the form $k_j\mathcal{W}_m$, $f_j\mathcal{W}_m$, and $c\mathcal{W}_m$.  But by invoking DeMorgan's laws and the identity $f_jc=f_j$, the inductive hypothesis implies the following inclusions:\\

\begin{align*}
k_j\mathcal{W}_m &\subseteq k_j\mathcal{KF}_n^0\cup k_jc\mathcal{KF}_n^0\\
&= k_j\mathcal{KF}_n^0\cup ci_j\mathcal{KF}_n^0 = \mathcal{KF}_n^0\cup c\mathcal{KF}_n^0;\\
\\
f_j\mathcal{W}_m &\subseteq f_j\mathcal{KF}_n^0\cup f_jc\mathcal{KF}_n^0\\
&= f_j\mathcal{KF}_n^0\cup f_j\mathcal{KF}_n^0 = \mathcal{KF}_n^0;\\
\\
c\mathcal{W}_m &\subseteq c\mathcal{KF}_n^0\cup cc\mathcal{KF}_n^0 = \mathcal{KF}_n^0\cup c\mathcal{KF}_n^0;
\end{align*}
\vspace{.3cm}

which concludes the inductive step and the proof.
\end{proof}

By the previous proposition, we are now justified in referring to the elements of $\mathcal{KF}_n^0$ as the \textit{even operators}, and those in $c\mathcal{KF}_n^0$ as the \textit{odd operators}.  By direct algebraic manipulation, it is easy to see that any operator in $\mathcal{KF}_n$ may be rewritten as a word in which the generator $c$ appears either zero times (the even case) or exactly one time (the odd case).  For example $k_1ci_1cck_1cf_1k_1c=k_1i_1f_1i_1$.

\begin{cor}  $\#\mathcal{KF}_n=2\cdot\#\mathcal{KF}_n^0$.
\end{cor}

In the special case $n=1$, the results of Gaida-Eremenko \cite{Gaida_Eremenko_1974a} imply that $\mathcal{KF}_1^0$ consists of no more than $17$ distinct even operators, which may be listed explicitly as below:\\

\begin{center} $\mathcal{KF}_1^0=\{\Id, k_1, i_1, k_1i_1, i_1k_1, i_1k_1i_1, k_1i_1k_1, f_1, f_1f_1, f_1k_1, f_1i_1, i_1f_1,$\\ $k_1i_1f_1, 0, f_1k_1i_1, f_1i_1k_1, f_1i_1f_1\}$.
\end{center}
\vspace{.3cm}

Adding $c$ to the left of each operator above yields the odd operators, for a total of $\#\mathcal{KF}_n\leq 34$.  The operators are indeed distinct when, for instance, $X=\mathbb{R}$ and $\tau_1$ is the usual topology on the reals, and in this case we get $\#\mathcal{KF}_n=34$.

We are ready to state some elementary algebraic identities in $\mathcal{KF}_n^0$, which are easily proven.  The first one is prominent in the solution to Kuratowski's original closure-complement problem.

\begin{lemma} \label{lemma_basics_unsaturated}  In any $n$-topological space $(X,\tau_1,...,\tau_n)$,
\begin{enumerate}
		\item (Kuratowski) for each $1\leq x\leq n$, 
				\begin{center} \fbox{$k_xi_xk_xi_x=k_xi_x$} and \fbox{$i_xk_xi_xk_x=i_xk_x$;}
				\end{center}
		\item for each $1\leq x,y\leq n$, 
				\begin{center} \fbox{$k_xk_y=k_{\max(x,y)}$} and \fbox{$i_xi_y=i_{\max(x,y)}$;}
				\end{center}
		\item for each $1\leq x, y\leq n$, 
				\begin{center} \fbox{if $x\leq y$ then $k_xf_y=f_y$.}
				\end{center}
\end{enumerate}
\end{lemma}

Recall that an $n$-topological space $(X,\tau_1,...,\tau_n)$ is \textit{saturated} if whenever $1\leq x,y\leq n$ and $U$ is a nonempty $\tau_x$-open set, then $i_y U\neq\emptyset$.  For the remainder of the paper, we assume that our space $(X,\tau_1,...,\tau_n)$ is saturated.  The most basic and important identity, which we use extensively, is proven in \cite{Banakh_2018a}:

\begin{lemma}[Banakh, Chervak, Martynyuk, Pylypovych, Ravsky, Simkiv] \label{lemma_banakh} Let $(X,\tau_1,...,\tau_n)$ be a saturated $n$-topological space.  For each $1\leq x,y\leq n$, $k_xi_y=k_xi_x$ and $i_xk_y=i_xk_x$.
\end{lemma}

This identity means that, assuming saturation, the second index in a word of the form $k_xi_y$ or $i_xk_y$ is irrelevant in determining the action of the operator.  For this reason, we find it convenient to adopt a \textit{star notation}, and simply write\\

\begin{center} for each $1\leq x,y\leq n$, \fbox{$k_xi_y=k_xi_*$} and \fbox{$i_xk_y=i_xk_*$.}
\end{center}
\vspace{.3cm}

We employ this notation in the following lemma.

\begin{lemma}[IF Lemma] \label{lemma_if}  Let $(X,\tau_1,...,\tau_n)$ be a saturated $n$-topological space.  For each $1\leq x,y\leq n$, 

\begin{center} \fbox{$i_xf_y=i_xf_*$.}
\end{center}
\end{lemma}

\begin{proof}  Since interiors distribute over intersections, by Lemma \ref{lemma_banakh} we have $i_xf_y=i_xk_y\wedge i_xk_y c=i_xk_*\wedge i_xk_*c=i_xf_*$.
\end{proof}

For other types of words, as below, it turns out that the value of $y$ is irrelevant if $y\leq x$, but may matter if $y>x$.

\begin{lemma}[FK Lemma] \label{lemma_fk}  Let $(X,\tau_1,...,\tau_n)$ be a saturated $n$-topological space.  For each $1\leq x,y\leq n$, 

\begin{center} \fbox{$f_xk_y=f_xk_{\max(x,y)}$.}
\end{center}
\end{lemma}

\begin{proof}  If $y\geq x$ then the statement is trivial.  Otherwise $y<x$, and we compute using Lemmas \ref{lemma_basics_unsaturated} and \ref{lemma_banakh} that $f_xk_y=k_xk_y\wedge ci_xk_y=k_x\wedge ci_xk_*=f_xk_x=f_xk_{\max(x,y)}$.
\end{proof}

For many of our algebraic lemmas involving $k_x$ or $i_x$, we may use DeMorgan's law to instantly deduce a ``dual'' corollary.

\begin{lemma}[FI Lemma]  \label{lemma_fi}  Let $(X,\tau_1,...,\tau_n)$ be a saturated $n$-topological space.  For each $1\leq x,y\leq n$, 

\begin{center} \fbox{$f_xi_y=f_xi_{\max(x,y)}$.}
\end{center}
\end{lemma}

\begin{proof}  By duality: $f_xi_y=f_xck_yc=f_xk_yc=f_xk_{\max(x,y)}c=f_xci_{\max(x,y)}=f_xi_{\max(x,y)}$.
\end{proof}

\begin{lemma}[FKF Lemma] \label{lemma_fkf}  Let $(X,\tau_1,...,\tau_n)$ be a saturated $n$-topological space.  Then for each $1\leq x,y,z\leq n$,

\begin{center} \fbox{if $y\leq\max(x,z)$, then $f_xk_yf_z=f_xf_z$.}
\end{center}
\end{lemma}

\begin{proof}  If $y\leq z$ then $k_yf_z=f_z$ by Lemma \ref{lemma_basics_unsaturated}.  Otherwise $y\leq x$, in which case we compute\\

\begin{align*}
f_xk_yf_z &= k_xk_yf_z\wedge ci_xk_yf_z\\
&= k_xf_z\wedge ci_xk_*f_z\\
&= k_xf_z\wedge ci_xf_z = f_xf_z.
\end{align*}
\end{proof}

\begin{lemma}[FIKI/FKIK/FKIF Lemma] \label{lemma_fiki}  Let $(X,\tau_1,...,\tau_n)$ be a saturated $n$-topological space.  For each $1\leq x,y\leq n$,
\begin{itemize}
		\item \fbox{if $y\leq x$, then $f_xi_yk_*i_*=f_xk_xi_*$.}
		\item \fbox{if $y\leq x$, then $f_xk_yi_*k_*=f_xi_xk_*$.}
		\item \fbox{if $y\leq x$, then $f_xk_yi_*f_*=f_xi_xf_*$.}
\end{itemize}
\end{lemma}

\begin{proof}  For the first item, by Lemmas \ref{lemma_basics_unsaturated} and \ref{lemma_banakh}, compute\\

\begin{align*}
f_xi_yk_*i_* &= k_xi_*k_*i_*\wedge k_xci_yk_*i_*\\
&= k_xi_*\wedge ci_xi_yk_*i_*\\
&= k_xk_xi_*\wedge ci_xk_xi_*\\
&= f_xk_xi_*.
\end{align*}
\vspace{.3cm}

The second item follows from the first by duality.  The third also follows from the first, by observing that $f_xk_yi_*f_*=f_xk_yi_*k_*f_*=f_xi_xk_*f_*=f_xi_xf_*$.
\end{proof}

The next lemma is a generalization of Gaida-Eremenko's observation, together with its dual statement.

\begin{lemma}[IFK/IFI Lemma] \label{lemma_zero}  Let $(X,\tau_1,...,\tau_n)$ be a saturated $n$-topological space.
\begin{itemize}
		\item For any $1\leq x,y,z\leq n$, \fbox{$i_x f_y k_z=0$.}
		\item For any $1\leq x,y,z\leq n$, \fbox{$i_x f_y i_z=0$.}
\end{itemize}
\end{lemma}

\begin{proof}  It suffices to prove that $i_nf_yk_z=0$, for if there existed a set $A\subseteq X$ with $i_xf_y k_zA\neq\emptyset$, then by saturation, we would have $i_nf_yk_zA=i_ni_xf_yk_zA\neq\emptyset$, which would contradict $i_nf_yk_z=0$.

We can use Lemma \ref{lemma_if} to rewrite $i_nf_yk_z=i_nf_*k_z=i_nf_nk_z$.  Then use Lemma \ref{lemma_fk} to write $i_nf_yk_z=i_nf_nk_n=0$.
\end{proof}

\begin{lemma}[FFK/FFI/FFF Lemma] \label{lemma_ffk}  Let $(X,\tau_1,...,\tau_n)$ be a saturated $n$-topological space.  For each $1\leq x,y,z\leq n$, the following hold.
\begin{itemize}
		\item \fbox{$f_xf_yk_z=k_xf_yk_z$.}
		\item \fbox{If $x\leq y$, then $f_xf_yk_z=f_yk_z$.}
		\item \fbox{$f_xf_yi_z=k_xf_yi_z$.}
		\item \fbox{If $x\leq y$, then $f_xf_yi_z=f_yi_z$.}
		\item \fbox{$f_xf_yf_z=k_xf_yf_z$.}
		\item \fbox{If $x\leq y$, then $f_xf_yf_z=f_yf_z$.}
\end{itemize}
\end{lemma}

\begin{proof}  It suffices to prove the first statement, as the second follows immediately; the third and fourth follow from duality; and the fifth and sixth follow from the observation that $f_xf_yf_z=f_xf_yk_zf_z$.

Using Lemma \ref{lemma_zero}, we compute\\

\begin{align*}
		f_xf_yk_z &= k_xf_yk_z\wedge k_xcf_yk_z\\
    &= k_xf_yk_z\wedge ci_xf_yk_z\\
		&= k_xf_yk_z\wedge c0\\
		&= k_xf_yk_z\wedge 1 = k_xf_yk_z.
\end{align*}
\end{proof}

\begin{lemma}[FKFK/FKFI Lemma] \label{lemma_fkfk}  Let $(X,\tau_1,...,\tau_n)$ be a saturated $n$-topological space.
\begin{itemize}
		\item For any $1\leq x,y,z,w\leq n$, \fbox{$f_xk_yf_zk_w=k_{\max(x,y)}f_zk_w$.}
		\item For any $1\leq x,y,z,w\leq n$, \fbox{$f_xk_yf_zi_w=k_{\max(x,y)}f_zi_w$.}
\end{itemize}
\end{lemma}

\begin{proof}  Using Lemma \ref{lemma_zero} again,\\

\begin{align*}
f_xk_yf_zk_w &= k_xk_yf_zk_w \wedge ci_xk_yf_zk_w\\
&= k_{\max(x,y)}f_zk_w\wedge ci_xk_*f_zk_w\\
&= k_{\max(x,y)}f_zk_w\wedge ci_xf_zk_w\\
&= k_{\max(x,y)}f_zk_w\wedge c0\\
&= k_{\max(x,y)}f_zk_w\wedge 1 = k_{\max(x,y)}f_zk_w.
\end{align*}
\end{proof}

\section{The Case of Two Topologies}

In this section we look closely at the special case where $n=2$, and solve the closure-complement-frontier problem for a saturated $2$-topological space.  The prototypical example is $(\mathbb{R},\tau_s,\tau_u)$ where $\tau_s=$ the Sorgenfrey topology (in which basic open neighborhoods have the form $[a,b)=\{x\in\mathbb{R}:a\leq x<b\}$) and $\tau_u=$ the usual Euclidean topology.

It is instructive to use Lemmas \ref{lemma_basics_unsaturated} through \ref{lemma_fkfk} to write out the distinct elements of $\mathcal{KF}_2^0$ explicitly.  There turn out to be at most $60$ of them.  This is an enjoyable computation and we postpone the details until the more general case of Section 4, where $n$ is arbitrary.  The reader may verify the truth of the following proposition by observing that applying any of the generators $k_x$, $i_x$, or $f_x$ ($x=1,2$) to the left of any of the $60$ words listed below will always simply produce another word on the list, and thus the entire monoid $\mathcal{KF}_2^0$ is accounted for.

\begin{prop}  The monoid $\mathcal{KF}_2^0$ consists of at most $60$ elements, which are listed in the table below.  Consequently, the monoid $\mathcal{KF}_2$ consists of at most $120$ elements.

\begin{center}
\begin{tabular}{|c|p{11.2cm}|c|}
    \hline
    Word Length & Operators & Count \\
    \hline
    0 & $\Id$ & 1\\
    \hline
    1 & $i_1,i_2,$ \hspace{0.6cm} $k_1,k_2,$ \hspace{0.6cm} $f_1,f_2$ & 6\\
    \hline
    2 & $k_1i_*,k_2i_*,$ \hspace{0.6cm} $i_1k_*, i_2k_*,$ \hspace{0.6cm} $f_1i_1, f_1i_2, f_2i_2,$ \hspace{0.6cm} $i_1f_*, i_2f_*,$\newline $f_1k_1, f_1k_2, f_2k_2,$ \hspace{0.6cm} $k_2f_1,$ \hspace{0.6cm} $f_1f_1, f_1f_2, f_2f_1, f_2f_2$ & 17\\
    \hline
    3 & $i_1k_*i_*, i_2k_*i_*,$ \hspace{0.6cm} $k_1i_*k_*, k_2i_*k_*,$ \hspace{0.6cm} $f_1k_1i_*, f_1k_2i_*, f_2k_2i_*,$ \newline $f_1i_1k_*, f_1i_2k_*, f_2i_2k_*,$ \hspace{0.6cm} $0,$ \hspace{0.6cm} $k_2f_1i_1, k_2f_1i_2,$ \newline $k_1i_*f_*, k_2i_*f_*,$ \hspace{0.6cm} $k_2f_1k_1, k_2f_1k_2,$ \hspace{0.6cm} $f_1k_2f_1,$ \newline $k_2f_1f_1, k_2f_1f_2,$ \hspace{0.6cm} $f_1i_1f_*,f_1i_2f_*, f_2i_2f_*$ & 23\\
    \hline
    4 & $f_1i_2k_*i_*,$ \hspace{0.6cm} $f_1k_2i_*k_*,$ \hspace{0.6cm} $k_2f_1k_1i_*, k_2f_1k_2i_*,$ \hspace{0.6cm} $k_2f_1i_1k_*,k_2f_1i_2k_*,$\newline $f_1k_2i_*f_*,$ \hspace{0.6cm} $k_2f_1i_1f_*, k_2f_1i_2f_*,$ \hspace{0.6cm} $k_2f_1k_2f_1$ & 10\\
    \hline
    5 & $k_2f_1k_2i_{*}k_{*}$, \hspace{0.6cm} $k_2f_1i_2k_{*}i_{*}$, \hspace{0.6cm} $k_2f_1k_2i_{*}f_{*}$ & 3\\
    \hline
\end{tabular}
\end{center}
\end{prop}

It is also straightforward to check, on a case-by-case basis, that the $60$ operators in $\mathcal{KF}_2^0$ are distinct, in the sense that for any $\omega_1,\omega_2$ as in the table above with $\omega_1\neq\omega_2$, there exists a subset $A^{\omega_1,\omega_2}$ of some $2$-topological space $(X,\tau_1,\tau_2)$ for which $\omega_1A^{\omega_1,\omega_2}\neq \omega_2A^{\omega_1,\omega_2}$.

Combining this observation with the simple lemma below, we obtain the stronger fact that there exists a $2$-topological space with an initial subset $A$ which distinguishes all of the operators in $\mathcal{KF}_2$ simultaneously.

\begin{lemma} \label{lemma_disjoint_union}  Suppose that for any distinct pair of operators $\omega_1,\omega_2\in\mathcal{KF}_n^0$, there exists a saturated $n$-topological space $X^{\omega_1,\omega_2}$ and a subset $A^{\omega_1,\omega_2}\subseteq X^{\omega_1,\omega_2}$ in which $\omega_1A^{\omega_1,\omega_2}\neq \omega_2A^{\omega_1,\omega_2}$.  Then there exist a saturated $n$-topological space $X$ and a subset $A\subseteq X$ such that $\omega_1A\neq \omega_2A$, for each pair of distinct operators $\omega_1,\omega_2\in\mathcal{KF}_n^0$.
\end{lemma}

\begin{proof}  If the assumption is true, then we can construct the $n$-topological disjoint union $X=\displaystyle\bigcup_{\substack{\omega_1,\omega_2\in\mathcal{KF}_n^0 \\ \omega_1\neq \omega_2}}X^{\omega_1,\omega_2}$ and form the initial set $A=\displaystyle\bigcup_{\omega_1\neq \omega_2}A^{\omega_1,\omega_2}$.  Then for any operators $\omega_1\neq\omega_2$ in $\mathcal{KF}_n^0$, we have $(\omega_1A)\Delta(\omega_2A)\supseteq(\omega_1A^{\omega_1,\omega_2})\Delta(\omega_2A^{\omega_1,\omega_2})\neq\emptyset$ (where $\Delta$ denotes the symmetric difference), and therefore $\omega_1A\neq\omega_2A$.
\end{proof}

Despite the preceding, we would like to follow the tradition of the closure-complement theorem by exhibiting an explicit initial set $A\subseteq\mathbb{R}$ which simultaneously distinguishes the operators in $\mathcal{KF}_2$.

\begin{example}[An Initial Set For $\mathcal{KF}_2$ in the Usual/Sorgenfrey Line] \label{ex_usual_sorg}  We consider the $2$-topological space $(\mathbb{R},\tau_1,\tau_2)$ where $\tau_1=\tau_s$ is the Sorgenfrey topology and $\tau_2=\tau_u$ is the usual Euclidean topology.  We define\\

 $$S^0 = \bigcup\limits_{k=0}^{\infty} \left(\frac{1}{3^{2k+1}}, \frac{1}{3^{2k}}\right), S^1 = \bigcup\limits_{k=0}^{\infty} \left[\frac{1}{3^{2k+1}}, \frac{1}{3^{2k}}\right), S^2 = \bigcup\limits_{k=0}^{\infty} \left[\frac{1}{3^{2k+1}}, \frac{1}{3^{2k}}\right], S^* = \bigcup\limits_{k=0}^{\infty} \left(\frac{1}{3^{2k+1}}, \frac{1}{3^{2k}}\right],$$

$$T^0=\bigcup\limits_{k=1}^{\infty} \left(\frac{1}{3^{2k}}, \frac{2}{3^{2k}}\right), T^1=\bigcup\limits_{k=1}^{\infty} \left[\frac{1}{3^{2k}},\; \frac{2}{3^{2k}}\right),  T^2=\bigcup\limits_{k=1}^{\infty} \left[\frac{1}{3^{2k}}, \frac{2}{3^{2k}}\right],  T^*=\bigcup\limits_{k=1}^{\infty} \left(\frac{1}{3^{2k}}, \frac{2}{3^{2k}}\right],$$
\vspace{.3cm}

\noindent and we take the following initial set:

\begin{center}
    \begin{align*}
        A &= \left(S^0 \cap \mathbb{Q}\right) \cup T^1 \cup \left(\left(2-S^0\right) \cap \mathbb{Q}\right) \cup \left(2-T^0\right) \cup \left(\left(2,3\right) \cap \mathbb{Q}\right) \cup \{4\} \cup (5,6) \cup (6,7) \\& \cup \bigcup\limits_{n=0}^{\infty} \left(\left(\frac{1}{2^{n+2}}S^2 + 8 - \frac{1}{2^n}\right)\cap \mathbb{Q}\right) \cup \bigcup\limits_{n=0}^{\infty} \left(\frac{1}{2^{n+2}}S^2 + 10 - \frac{1}{2^n}\right) \cup (10,11).
    \end{align*}
\end{center}
\vspace{.3cm}

It is possible to verify by hand that applying the $60$ operators of the monoid $\mathcal{KF}_2^0$ to $A$ yields $60$ distinct sets.  The results of such a computation appear in a previous draft of this paper (posted August 3, 2019) accessible via \url{arXiV.org}.  Bowron, in private communication, has also provided us with an elegant and brief computer-assisted verification.  Rather than presenting such a verification here, we will turn to a stronger result, by first considering the natural partial order on the monoid $\mathcal{KF}_n^0$.
\end{example}

The partial order is defined as follows: for every $o_1,o_2\in\mathcal{KF}_n^0$,\\

\begin{center} $o_1\leq o_2$ \hspace{0.6cm} if and only if \hspace{0.6cm} $o_1A\subseteq o_2A$ for every $A\subseteq X$.
\end{center}
\vspace{.3cm}

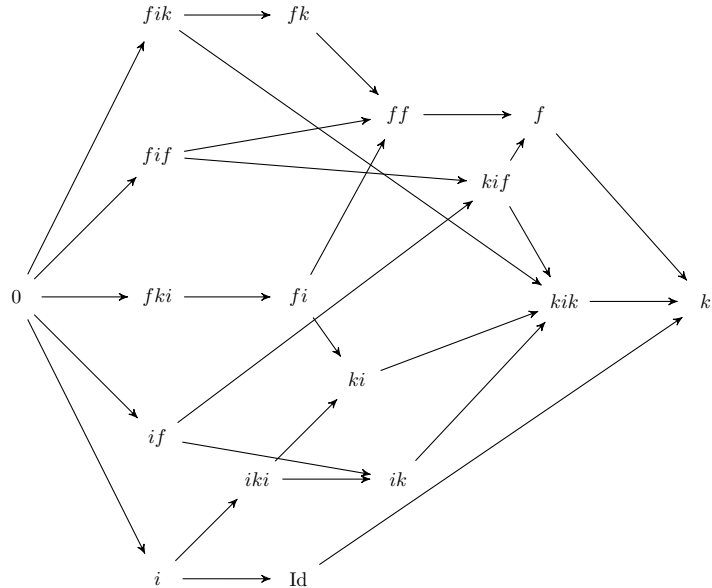
\begin{figure}[b] \label{figure_KF1}
\begin{tikzpicture}[->,>=stealth',shorten >=1pt,auto,node distance=2.5cm,
        scale = .75,transform shape, state without output/.append style={draw=none}]
\node[state] (ifk) [] {$0$};
\node[state] (fki) [right of=ifk] {$fki$};
\node[state] (fif) [above of=fki] {$fif$};
\node[state] (fik) [above of=fif] {$fik$};
\node[state] (if) [below of=fki] {$if$};
\node[state] (fi) [right of=fki] {$fi$};
\node[state] (fk) [right of=fik] {$fk$};
\node[state] (ff) [below right of=fk] {$ff$};
\node[state] (kif) [position=-35:{1.2cm} from ff] {$kif$};
\node[state] (f) [right of=ff] {$f$};

\node[state] (i) [below of=if] {$i$};
\node[state] (Id) [right of=i] {$\Id$};
\node[state] (iki) [above right of=i] {$iki$};
\node[state] (ik) [right of=iki] {$ik$};
\node[state] (ki) [above right of=iki] {$ki$};
\node[state] (kik) [position=-60:{1.5cm} from kif] {$kik$};
\node[state] (k) [right of=kik] {$k$};

\path
(ifk) edge              node {} (fki)
(ifk) edge              node {} (fik)
(ifk) edge              node {} (if)
(ifk) edge              node {} (fif)
(ifk) edge							node {} (i)
(fki) edge              node {} (fi)
(fik) edge              node {} (fk)
(fik) edge							node {} (kik)
(if) edge              node {} (kif)
(if) edge 						node {} (ik)
(fif) edge              node {} (ff)
(fif) edge              node {} (kif)
(fi) edge              node {} (ff)
(fi) edge							node {} (ki)
(fk) edge              node {} (ff)
(ff) edge              node {} (f)
(kif) edge              node {} (f)
(kif) edge							node {} (kik)
(i) edge 								node {} (iki)
(i) edge								node {} (Id)
(Id) edge								node {} (k)
(iki) edge							node {} (ki)
(iki) edge							node {} (ik)
(ki) edge								node {} (kik)
(ik) edge								node {} (kik)
(kik) edge							node {} (k)
(f) edge							node {} (k);

\end{tikzpicture}
\caption{The partial ordering on the $17$ operators of $\mathcal{KF}_1^0$, which was computed by Gaida and Eremenko but did not appear in the printed version of their article \cite{Gaida_Eremenko_1974a}; see also \cite{Eremenko_2020a}.  Subscripts are omitted from the notation since only one topology is involved.}
\end{figure}

The partial orderings on $\mathcal{K}_1^0$, $\mathcal{KF}_1^0$ (see Figure 1), and other related monoids have been diagrammed by various authors; see especially \cite{Gardner_Jackson_2008a} and \cite{Eremenko_2020a}.  It is clear that $\mathcal{KF}_n$ has a minimal element $0$ and a maximal element $k_n$, and that $0\leq i_n\leq...\leq i_1\leq \Id\leq k_1\leq ...\leq k_n$.  It is also clear that for any set operator $o$ we have $i_jo\leq o\leq k_jo$.

By the definition, for any operators $o_1,o_2,o_3$, if $o_1\leq o_2$ then $o_1o_3\leq o_2o_3$, so order is preserved by multiplication on the right.  The operators $i_j$ and $k_j$ ($1\leq j\leq n$) are also left order-preserving in the sense that if $o_1\leq o_2$, then $i_jo_1\leq i_jo_2$ and $k_jo_1\leq k_jo_2$.  On the other hand, $f_j$ is not left order-preserving in general.

\begin{example}[Exhibiting the Partial Order on $\mathcal{KF}_2^0$] \label{example_bowron}

We will now show there exists a set $A$ in a $2$-topological space with the property that $o_1\leq o_2$ if and only if $o_1A\subseteq o_2A$, for each $o_1,o_2\in\mathcal{KF}_2^0$.  In particular, the $60$ operators of $\mathcal{KF}_2^0$ applied to $A$ yield $60$ distinct sets.

We first present a list of apparently non-obvious inequalities in the partially ordered set $\mathcal{KF}_2^0$.

\begin{prop}  The following relations hold in any saturated $2$-topological space $(X,\tau_1,\tau_2)$:
\begin{enumerate}[(a)]
    \item $f_1i_1 \leq f_1i_2$ and $f_1k_1 \leq f_1k_2$;
    \item $f_1k_1i_* \leq f_1i_2k_*i_*$ and $f_1i_1k_* \leq f_1k_2i_*k_*$;
    \item $f_1f_1 \leq f_1k_2f_1$;
		\item $f_1k_2f_1\leq f_1f_2$;
		\item $f_1k_2i_*k_*\leq f_1k_2$ and $f_1i_2k_*i_*\leq f_1i_2$;
		\item $f_1k_2\leq f_1f_2$.
\end{enumerate}
\end{prop}

\begin{proof}  For (a), we have $f_1i_1=k_1i_*\wedge ci_1=k_1i_*\wedge k_1c\leq k_1i_*\wedge k_2c= k_1i_*\wedge ci_2=f_1i_2$, and the second statement follows in a dual way, because we can multiply the first inequality on the right by $c$, and get $f_1k_1=f_1i_1c\leq f_1i_2c=f_1k_2$.


For (b), we have

\begin{align*}
f_1k_1i_* &= k_1k_1i_* \wedge k_1ck_1i_* = k_1i_* \wedge k_1i_*k_*c.\\
f_1i_2k_*i_* &= k_1i_*k_*i_* \wedge k_1ci_2k_*i_* = k_1i_* \wedge k_2i_*k_*c.
\end{align*}
\vspace{.3cm}

The second statement follows dually.

For (c), we compute $f_1f_1=f_1\wedge ci_1f_1$ and $f_1k_2f_1=k_2f_1\wedge ci_1k_*f_1=k_2f_1\wedge ci_1f_1$, so the inequality $f_1f_1\leq f_1k_2f_1$ follows from $f_1\leq k_2f_1$.

For (d), compute $f_1f_1=f_1\wedge ci_1f_1$ and $f_1k_2f_1=k_2f_1\wedge ci_1k_*f_1=k_2f_1\wedge ci_1f_1$, so the inequality $f_1f_1\leq f_1k_2f_1$ follows from $f_1\leq k_2f_1$.

For (e), compute $f_1k_2i_*k_*=k_1k_2i_*k_*\wedge k_1ck_2i_*k_*=k_2i_*k_*\wedge k_1i_*c\leq k_2\wedge k_1ck_2=f_1k_2$.

Lastly, for (f), note that $k_1i_*\leq k_1\leq k_2$. Hence
$f_1k_2=k_1k_2\wedge k_1ck_2
=k_2\wedge k_1i_*c
=k_2\wedge(k_2c\wedge k_1i_*c)
\leq\left[(k_2\wedge k_2c)\wedge k_1i_*c\right]\vee\left[(k_2\wedge k_2c)\wedge k_1i_*\right]
=(k_2\wedge k_2c)\wedge (k_1i_*c\vee k_1i_*)
=f_2\wedge (k_1ck_*\vee k_1ck_*c)=f_2\wedge k_1(ck_2\vee ck_2c)
=k_1f_2\wedge k_1c(k_2\wedge k_2c)
=f_1f_2$.
\end{proof}

Using the inequalities in the proposition, together with the facts that closure and interior are left order-preserving, and all operators are right order-preserving, we obtain the diagram of the partially ordered set $\mathcal{KF}_2^0$ depicted in Figure 2.

\afterpage{\clearpage}

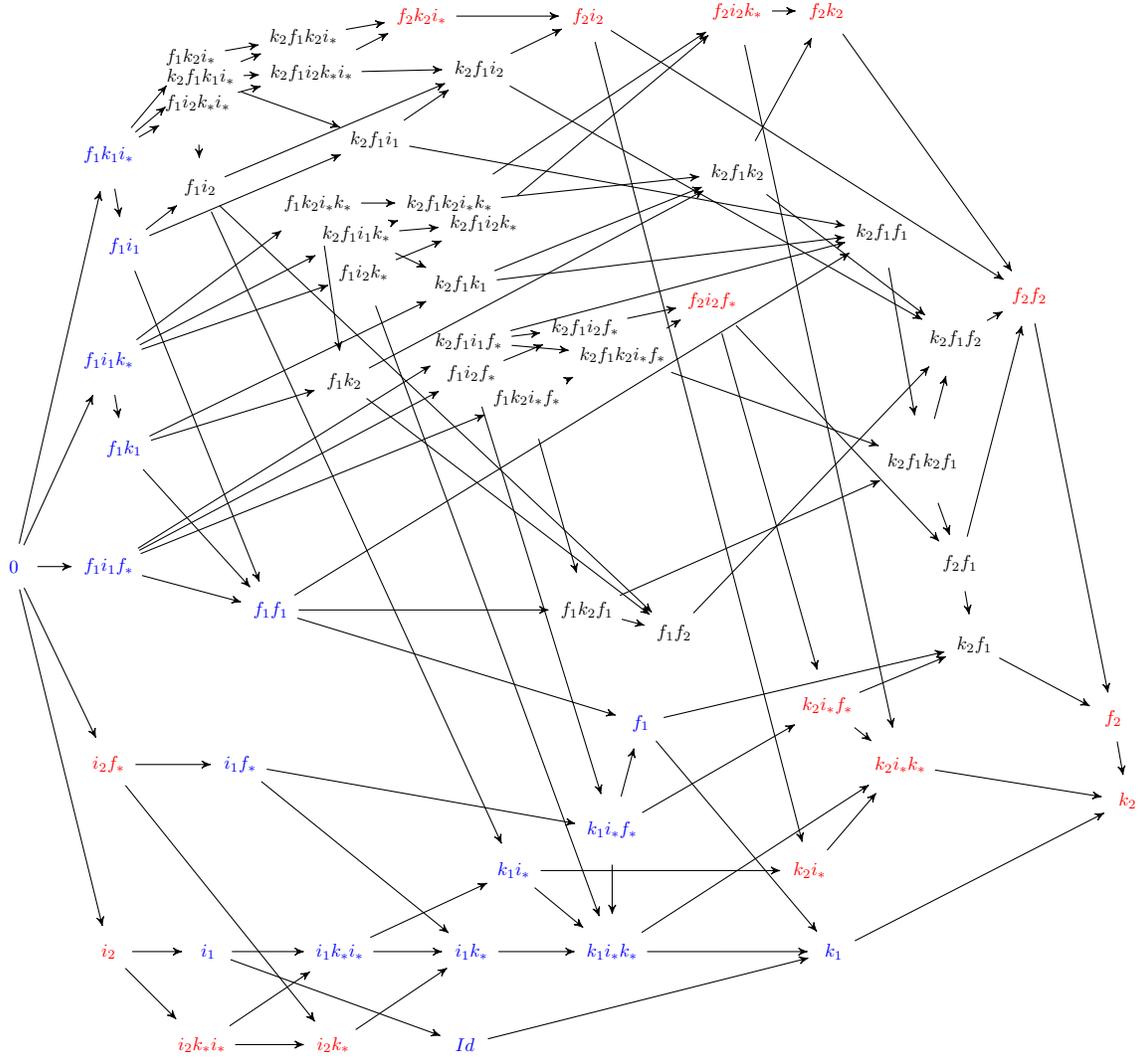
\begin{figure}[p] \label{figure_KF2}
\begin{tikzpicture}[->,>=stealth',shorten >=1pt,auto,node distance=2.5cm,
        scale = 0.70,transform shape, state without output/.append style={draw=none}]

  \node[state] (0) [] {\textcolor{blue}{$0$}};
	
  \node[state] (f_1i_1f_*) [position=0:{0.7cm} from 0] {\textcolor{blue}{$f_1i_1f_*$}};
  \node[state] (f_1i_1k_*) [position=90:{2.6cm} from f_1i_1f_*] {\textcolor{blue}{$f_1i_1k_*$}};
	\node[state] (f_1k_1i_*) [position=90:{2.6cm} from f_1i_1k_*] {\textcolor{blue}{$f_1k_1i_*$}};
  \node[state] (i_2f_*) [position=-90:{2.6cm} from f_1i_1f_*] {\textcolor{red}{$i_2f_*$}};
  \node[state] (i_2) [position=-90:{2.6cm} from i_2f_*] {\textcolor{red}{$i_2$}};
	
	\node[state] (f_1i_2k_*i_*) [position=30:{0.5cm} from f_1k_1i_*] {$f_1i_2k_*i_*$};
  \node[state] (k_2f_1k_1i_*) [position=40:{0.8cm} from f_1k_1i_*] {$k_2f_1k_1i_*$};
  \node[state] (f_1k_2i_*) [position=50:{1.1cm} from f_1k_1i_*] {$f_1k_2i_*$};  
	\node[state] (k_2f_1k_2i_*) [position=10:{7mm} from f_1k_2i_*] {$k_2f_1k_2i_*$};
  \node[state] (k_2f_1i_2k_*i_*) [position=15:{5mm} from f_1i_2k_*i_*] {$k_2f_1i_2k_*i_*$};
  \node[state] (f_2k_*i_*) [position=10:{8mm} from k_2f_1k_2i_*] {\textcolor{red}{$f_2k_2i_*$}};
  \node[state] (f_1i_1) [position=-80:{6mm} from f_1k_1i_*] {\textcolor{blue}{$f_1i_1$}};
  \node[state] (f_1i_2) [position=38:{0.8cm} from f_1i_1] {$f_1i_2$};
  \node[state] (k_2f_1i_1) [position=23:{4cm} from f_1i_1] {$k_2f_1i_1$};
  \node[state] (k_2f_1i_2) [position=23:{4.6cm} from f_1i_2] {$k_2f_1i_2$};
  \node[state] (f_2i_*) [position=0:{2cm} from f_2k_*i_*] {\textcolor{red}{$f_2i_2$}};
	
  \node[state] (f_1i_2k_*) [position=19:{3.8cm} from f_1i_1k_*] {$f_1i_2k_*$};
  \node[state] (k_2f_1i_1k_*) [position=27:{3.8cm} from f_1i_1k_*] {$k_2f_1i_1k_*$};
  \node[state] (f_1k_2i_*k_*) [position=37:{3.5cm} from f_1i_1k_*] {$f_1k_2i_*k_*$};
  \node[state] (k_2f_1i_2k_*) [position=5:{0.8cm} from k_2f_1i_1k_*] {$k_2f_1i_2k_*$};
  \node[state] (k_2f_1k_2i_*k_*) [position=0:{0.7cm} from f_1k_2i_*k_*] {$k_2f_1k_2i_*k_*$};
  \node[state] (f_2i_*k_*) [position=40:{4.8cm} from k_2f_1i_2k_*] {\textcolor{red}{$f_2i_2k_*$}};
  \node[state] (f_1k_1) [position=-80:{5mm} from f_1i_1k_*] {\textcolor{blue}{$f_1k_1$}};
  \node[state] (f_1k_2) [position=17:{3.3cm} from f_1k_1] {$f_1k_2$};
  \node[state] (k_2f_1k_1) [position=26:{5.9cm} from f_1k_1] {$k_2f_1k_1$};
	\node[state] (k_2f_1k_2) [position=22:{4.3cm} from k_2f_1k_1] {$k_2f_1k_2$};
  \node[state] (f_2k_*) [position=0:{5mm} from f_2i_*k_*] {\textcolor{red}{$f_2k_2$}};
	
  \node[state] (f_1i_2f_*) [position=28:{6.5cm} from f_1i_1f_*] {$f_1i_2f_*$};
  \node[state] (k_2f_1i_1f_*) [position=32:{6.6cm} from f_1i_1f_*] {$k_2f_1i_1f_*$};
  \node[state] (f_1k_2i_*f_*) [position=22:{7.1cm} from f_1i_1f_*] {$f_1k_2i_*f_*$};
  \node[state] (k_2f_1i_2f_*) [position=7:{0.6cm} from k_2f_1i_1f_*] {$k_2f_1i_2f_*$};
  \node[state] (k_2f_1k_2i_*f_*) [position=24:{0.2cm} from f_1k_2i_*f_*] {$k_2f_1k_2i_*f_*$};
  \node[state] (f_2i_*f_*) [position=12:{1cm} from k_2f_1i_2f_*] {\textcolor{red}{$f_2i_2f_*$}};
  \node[state] (f_1f_1) [position=-15:{2cm} from f_1i_1f_*] {\textcolor{blue}{$f_1f_1$}};
  \node[state] (f_1k_2f_1) [position=0:{4.8cm} from f_1f_1] {$f_1k_2f_1$};
  \node[state] (k_2f_1f_1) [position=7:{6.7cm} from k_2f_1k_1] {$k_2f_1f_1$};
  \node[state] (k_2f_1k_2f_1) [position=-80:{2.9cm} from k_2f_1f_1] {$k_2f_1k_2f_1$};
	
  \node[state] (f_1f_2) [position=-15:{0.5cm} from f_1k_2f_1] {$f_1f_2$};
  \node[state] (k_2f_1f_2) [position=75:{0.9cm} from k_2f_1k_2f_1] {$k_2f_1f_2$};
  \node[state] (f_2f_2) [position=30:{0.4cm} from k_2f_1f_2] {\textcolor{red}{$f_2f_2$}};
	
	\node[state] (f_2f_1) [position=-70:{7mm} from k_2f_1k_2f_1] {$f_2f_1$};

  \node[state] (i_1) [position=0:{1cm} from i_2] {\textcolor{blue}{$i_1$}};
  \node[state] (i_2k_*i_*) [below right of=i_2] {\textcolor{red}{$i_2k_*i_*$}};
  \node[state] (i_1k_*i_*) [right of=i_1] {\textcolor{blue}{$i_1k_*i_*$}};
  \node[state] (i_1f_*) [right of=i_2f_*] {\textcolor{blue}{$i_1f_*$}};
  \node[state] (k_1i_*f_*) [position=-10:{6cm} from i_1f_*] {\textcolor{blue}{$k_1i_*f_*$}};
  \node[state] (i_1k_*) [right of=i_1k_*i_*] {\textcolor{blue}{$i_1k_*$}};
  \node[state] (i_2k_*) [right of=i_2k_*i_*] {\textcolor{red}{$i_2k_*$}};
  \node[state] (ID) [right of=i_2k_*] {\textcolor{blue}{$Id$}};
	
  \node[state] (f_1) [position=75:{1cm} from k_1i_*f_*] {\textcolor{blue}{$f_1$}};
	\node[state] (k_2f_1) [position=-80:{0.5cm} from f_2f_1] {$k_2f_1$};
  \node[state] (f_2) [position=-28:{2cm} from k_2f_1] {\textcolor{red}{$f_2$}};
  
  \node[state] (k_2i_*f_*) [position=30:{3.4cm} from k_1i_*f_*] {\textcolor{red}{$k_2i_*f_*$}};

  \node[state] (k_1i_*) [position=25:{2.5cm} from i_1k_*i_*] {\textcolor{blue}{$k_1i_*$}};
  \node[state] (k_2i_*) [position=0:{4.6cm} from k_1i_*] {\textcolor{red}{$k_2i_*$}};
  \node[state] (k_1i_*k_*) [position=0:{1.5cm} from i_1k_*] {\textcolor{blue}{$k_1i_*k_*$}};
 
  \node[state] (k_2i_*k_*) [position=33:{5.2cm} from k_1i_*k_*] {\textcolor{red}{$k_2i_*k_*$}};

  \node[state] (k_1) [position=0:{3.1cm} from k_1i_*k_*] {\textcolor{blue}{$k_1$}};
  \node[state] (k_2) [position=-80:{0.7cm} from f_2] {\textcolor{red}{$k_2$}};

  \path (0) edge              node {} (i_2)
        (0) edge              node {} (i_2f_*)
        (0) edge              node {} (f_1i_1f_*)
        (0) edge              node {} (f_1i_1k_*)
        (0) edge              node {} (f_1k_1i_*)
        (f_1k_2i_*f_*) edge              node {} (k_2f_1k_2i_*f_*)
        (f_1i_1f_*) edge              node {} (k_2f_1i_1f_*)
        (f_1i_1f_*) edge              node {} (f_1i_2f_*)
				(f_1i_1f_*) edge							node {} (f_1k_2i_*f_*)
        (i_1f_*) edge              node {} (k_1i_*f_*)
        (i_2) edge              node {} (i_1)
        (i_2) edge              node {} (i_2k_*i_*)
        (i_1) edge              node {} (i_1k_*i_*)
        (i_2k_*i_*) edge              node {} (i_1k_*i_*)
        (i_2f_*) edge              node {} (i_1f_*)
        (i_1f_*) edge              node {} (i_1k_*)
        (i_1k_*i_*) edge              node {} (i_1k_*)
        (i_2k_*i_*) edge              node {} (i_2k_*)
        (i_2k_*) edge              node {} (i_1k_*)
        (i_1) edge              node {} (ID)
        (f_1k_1i_*) edge              node {} (f_1i_2k_*i_*)
        (f_1i_2k_*i_*) edge              node {} (k_2f_1i_2k_*i_*)
        (f_1k_1i_*) edge              node {} (f_1k_2i_*)
        (f_1k_1i_*) edge              node {} (k_2f_1k_1i_*)
        (f_1k_2i_*k_*) edge              node {} (k_2f_1k_2i_*k_*)
        (f_1i_1k_*) edge              node {} (k_2f_1i_1k_*)
        (f_1i_1k_*) edge              node {} (f_1i_2k_*)
				(f_1i_1k_*) edge							node {} (f_1k_2i_*k_*)
        (k_2f_1i_1k_*) edge              node {} (k_2f_1i_2k_*)
        (f_1i_2k_*) edge              node {} (k_2f_1i_2k_*)
        (k_2f_1i_2k_*) edge              node {} (f_2i_*k_*)
        (k_2f_1k_2i_*k_*) edge              node {} (f_2i_*k_*)
        (k_2f_1i_2k_*i_*) edge              node {} (f_2k_*i_*)
        (f_1k_2i_*) edge              node {} (k_2f_1k_2i_*)
        (k_2f_1k_2i_*) edge              node {} (f_2k_*i_*)
        (k_2f_1k_1i_*) edge              node {} (k_2f_1k_2i_*)
        (k_2f_1i_1f_*) edge              node {} (k_2f_1i_2f_*)
        (f_1i_2f_*) edge              node {} (k_2f_1i_2f_*)
        (k_2f_1k_2i_*f_*) edge              node {} (f_2i_*f_*)
        (k_2f_1i_2f_*) edge              node {} (f_2i_*f_*)
        (f_1i_1k_*) edge              node {} (f_1k_1)
				
				(f_2i_*k_*) edge 	node {} (k_2i_*k_*)
				
        (f_1k_1) edge              node {} (f_1k_2)
        (f_1k_2) edge              node {} (k_2f_1k_2)
        (f_1k_1) edge              node {} (k_2f_1k_1)
        (k_2f_1k_1) edge              node {} (k_2f_1k_2)
        (k_2f_1k_2) edge              node {} (f_2k_*)
        (f_2i_*k_*) edge              node {} (f_2k_*)
        (f_1k_1i_*) edge              node {} (f_1i_1)
        (f_1i_1) edge              node {} (f_1i_2)
        (f_1i_1) edge              node {} (k_2f_1i_1)
        (k_2f_1i_1) edge              node {} (k_2f_1i_2)
        (f_1i_2) edge              node {} (k_2f_1i_2)
        (k_2f_1i_2) edge              node {} (f_2i_*)
        (f_2k_*i_*) edge              node {} (f_2i_*)
        (i_1k_*i_*) edge              node {} (k_1i_*)
        (f_2i_*f_*) edge              node {} (k_2i_*f_*)
        (k_1i_*f_*) edge              node {} (k_2i_*f_*)
        (f_1i_1) edge              node {} (f_1f_1)
        (f_1k_1) edge              node {} (f_1f_1)
        (f_1i_1f_*) edge              node {} (f_1f_1)
        (k_1i_*) edge              node {} (k_1i_*k_*)
        (k_1i_*f_*) edge              node {} (k_1i_*k_*)
        (i_1k_*) edge              node {} (k_1i_*k_*)
        (k_1i_*) edge              node {} (k_2i_*)
        (k_2i_*) edge              node {} (k_2i_*k_*)
        (k_1i_*k_*) edge              node {} (k_2i_*k_*)
        (k_2i_*f_*) edge              node {} (k_2i_*k_*)
        (f_1f_1) edge              node {} (k_2f_1f_1)
        (f_1f_1) edge              node {} (f_1)
        (k_1i_*f_*) edge              node {} (f_1)
        (f_1f_1) edge              node {} (f_1k_2f_1)
        (f_1f_2) edge              node {} (k_2f_1f_2)
        (f_1k_2f_1) edge              node {} (k_2f_1k_2f_1)
        (k_2f_1k_2f_1) edge              node {} (f_2f_1)
        (f_2f_1) edge              node {} (f_2f_2)
        (f_2k_*) edge              node {} (f_2f_2)
        (k_2f_1f_2) edge              node {} (f_2f_2)
        (f_2i_*) edge              node {} (f_2f_2)
        (f_2i_*f_*) edge              node {} (f_2f_1)
        (f_2f_1) edge              node {} (k_2f_1)
        (f_1) edge              node {} (k_2f_1)
        (f_2f_2) edge              node {} (f_2)
        (k_2f_1) edge              node {} (f_2)
        (f_1) edge              node {} (k_1)
        (k_1i_*k_*) edge              node {} (k_1)
        (ID) edge              node {} (k_1)
        (k_1) edge              node {} (k_2)
        (k_2i_*k_*) edge              node {} (k_2)
        (f_2) edge              node {} (k_2)
        (f_1i_2k_*) edge              node {} (k_1i_*k_*)
				(k_2f_1k_2) edge					node {} (k_2f_1f_2)
				(k_2f_1i_2) edge					node {} (k_2f_1f_2)
				(f_1k_2) edge							node {} (f_1f_2)
				(f_1i_2) edge							node {} (f_1f_2)
				
				
				(k_2f_1k_1) edge node {} (k_2f_1f_1)
				(f_1k_2f_1) edge node {} (f_1f_2)
				(k_2f_1f_1) edge node {} (k_2f_1k_2f_1)
				(f_1i_2) edge node {} (k_1i_*)
				(f_2i_*) edge node {} (k_2i_*)
				(i_2f_*) edge node {} (i_2k_*)
				(k_2f_1i_1) edge node {} (k_2f_1f_1)
				(f_1i_2f_*) edge node {} (k_1i_*f_*)
				(k_2i_*f_*) edge node {} (k_2f_1)
				
				(f_1k_2i_*f_*) edge node {} (f_1k_2f_1)
				(f_1i_2k_*i_*) edge node {} (f_1i_2)
				(k_2f_1k_2f_1) edge node {} (k_2f_1f_2)
				(f_1k_2i_*k_*) edge node {} (f_1k_2)
				(k_2f_1i_1k_*) edge node {} (k_2f_1k_1)
				(k_2f_1i_1k_*) edge node {} (k_2f_1k_2i_*k_*)
				(k_2f_1k_2i_*f_*) edge node {} (k_2f_1k_2f_1)
				(k_2f_1i_2k_*i_*) edge node {} (k_2f_1i_2)
				(k_2f_1k_2i_*k_*) edge node {} (k_2f_1k_2)
				
				(k_2f_1k_1i_*) edge node {} (k_2f_1i_1)
				(k_2f_1i_1f_*) edge node {} (k_2f_1f_1)
				(k_2f_1k_1i_*) edge node {} (k_2f_1i_2k_*i_*)
				(k_2f_1i_1f_*) edge node {} (k_2f_1k_2i_*f_*);
\end{tikzpicture}

\caption{The partial ordering on $\mathcal{KF}_2^0$.  The blue operators are operators that can be built using exclusively the $\tau_1$ topology. The red operators are operators built using the topology $\tau_2$ that cannot also be built using $\tau_1$. The black operators are those built using a combination of both topologies.}
\end{figure}

To show that no further inequalities hold in general, we define a partition
$P=\{P_0,\ldots,P_{12}\}$ of $\mathbb{R}^+\!$ such that for each inequality
$o_1\leq o_2$ ($o_1,o_2\in\mathcal{KF}_2^0$)
not implied by Figure 2, there exist integers
$0\leq\alpha_1<\dots<\alpha_n\leq12$ ($1\leq n\leq12$) satisfying
$o_1(P_{\alpha_1}\cup\dots\cup P_{\alpha_n})\not\subseteq
o_2(P_{\alpha_1}\cup\dots\cup P_{\alpha_n})$ in
$(\mathbb{R}^+\!,\tau_1,\tau_2)$ where $\tau_1=\tau_s$ is the Sorgenfrey topology
and $\tau_2=\tau_u$ is the usual Euclidean topology.

The partition $\{\pi_1,\ldots,\pi_8\}$
of $(0,1]$ is defined as follows:

\newcommand{\hs}{\vrule width10pt height0pt depth0pt}

\begin{equation*}
\begin{aligned}
\pi_1&=\textstyle\bigcup_{n=1}^\infty\left\{\frac{1}{3^{2n}}\right\}\\
\pi_2&=\textstyle\bigcup_{n=1}^\infty\left(\frac{1}{3^{2n}},\frac{2}{3^{2n}}\right)\\\\
\end{aligned}\hs\begin{aligned}
\pi_3&=\textstyle\bigcup_{n=1}^\infty\left\{\frac{2}{3^{2n}}\right\}\\
\pi_4&=\textstyle\bigcup_{n=1}^\infty\left(\frac{2}{3^{2n}},\frac{1}{3^{2n-1}}\right)\\\\
\end{aligned}\hs\begin{aligned}
\pi_5&=\textstyle\bigcup_{n=1}^\infty\left\{\frac{1}{3^{2n-1}}\right\}\\
\pi_6&=\textstyle\bigcup_{n=1}^\infty\left(\frac{1}{3^{2n-1}},\frac{1}{3^{2n-2}}\right)\cap\mathbb{Q}\\
\pi_7&=\textstyle\bigcup_{n=1}^\infty\left(\frac{1}{3^{2n-1}},\frac{1}{3^{2n-2}}\right)\setminus\mathbb{Q}\\
\end{aligned}\hs\begin{aligned}
\pi_{8}&=\{1\}.\\\\\\
\end{aligned}
\end{equation*}

\noindent
For $1\leq j\leq8$, let $P_j=\textstyle\bigcup_{n=0}^\infty\left(1-\frac{1}{2^n}+
\frac{1}{2^{n+2}}\pi_j\right)$. Thus

\[
\textstyle P_1\cup\cdots\cup P_8=(0,\frac{1}{4}]\cup(\frac{1}{2},\frac{5}{8}]\cup(\frac{3}{4},
\frac{13}{16}]\cup\cdots.
\]

\noindent
To complete the definition of $P$, set

\begin{equation*}
\begin{aligned}
P_0&=\left\{\textstyle1-\frac{1}{2^n}:n=1,2,\ldots\right\}\\
P_9&=\textstyle(\frac{1}{4},\frac{1}{2})\cup(\frac{5}{8},\frac{3}{4})\cup(\frac{13}{16},
\frac{7}{8})\cup\cdots\\
\end{aligned}\hs\begin{aligned}
P_{10}&=\{1\}\\\\
\end{aligned}\hs\hs\begin{aligned}
P_{11}&=(1,\infty)\cap\mathbb{Q}\\
P_{12}&=(1,\infty)\setminus\mathbb{Q}\,.\\
\end{aligned}
\end{equation*}

\noindent
Then each of the following equations holds in $(\mathbb{R}^+\!,\tau_1,\tau_2)$:
\begin{equation*}
\begin{aligned}
k_1P_0&=P_0\\
k_1P_1&=P_0\cup P_1\\
k_1P_2&=P_0\cup P_1\cup P_2\\
k_1P_3&=P_0\cup P_3\\
k_1P_4&=P_0\cup P_3\cup P_4\\
k_1P_5&=P_0\cup P_5\\
k_1P_6&=P_0\cup P_5\cup P_6\cup P_7\\
k_1P_7&=P_0\cup P_5\cup P_6\cup P_7\\
k_1P_8&=P_8\\
k_1P_9&=P_8\cup P_9\\
k_1P_{10}&=P_{10}\\
k_1P_{11}&=P_{10}\cup P_{11}\cup P_{12}\\
k_1P_{12}&=P_{10}\cup P_{11}\cup P_{12}\end{aligned}\hs\begin{aligned}
k_2P_0&=P_0\cup P_{10}\\
k_2P_1&=P_0\cup P_1\cup P_{10}\\
k_2P_2&=P_0\cup P_1\cup P_2\cup P_3\cup P_{10}\\
k_2P_3&=P_0\cup P_3\cup P_{10}\\
k_2P_4&=P_0\cup P_3\cup P_4\cup P_5\cup P_{10}\\
k_2P_5&=P_0\cup P_5\cup P_{10}\\
k_2P_6&=P_0\cup P_1\cup P_5\cup P_6\cup P_7\cup P_8\cup P_{10}\\
k_2P_7&=P_0\cup P_1\cup P_5\cup P_6\cup P_7\cup P_8\cup P_{10}\\
k_2P_8&=P_8\cup P_{10}\\
k_2P_9&=P_0\cup P_8\cup P_9\cup P_{10}\\
k_2P_{10}&=P_{10}\\
k_2P_{11}&=P_{10}\cup P_{11}\cup P_{12}\\
k_2P_{12}&=P_{10}\cup P_{11}\cup P_{12}\,.\\
\end{aligned}
\end{equation*}

\noindent
Using these equations, all inclusions not implied by Figure 2 may be eliminated computationally.  Bowron has written the following C program and Python script which verify the eliminations:\\

\noindent
\url{https://github.com/mathematrucker/polytopological-spaces/blob/master/figure_2.c}

\noindent
\url{https://github.com/mathematrucker/polytopological-spaces/blob/master/figure_2.py}\\

Following Lemma \ref{lemma_disjoint_union}, we may take the disjoint union of all possible sets of the form $P_{\alpha_1}\cup\dots\cup P_{\alpha_n}$ to obtain an initial set $A$ with the property that if $o_1,o_2\in\mathcal{KF}_2^0$ and $o_1\not\leq o_2$, then $o_1A\not\subseteq o_2A$.  Consequently, $o_1\leq o_2$ if and only if $o_1A\subseteq o_2A$, for all $o_1,o_2\in\mathcal{KF}_2^0$.
\end{example}

\section{The General Case}

We are ready to solve the closure-complement-frontier problem in the general setting of a saturated $n$-topological space where $n$ is arbitrary.  The surprising fact which underlies our computation is that every reduced word in $\mathcal{KF}_n^0$ has length $\leq 5$, and in fact has the same form as one of the reduced words which we already computed in Section 3 for $\mathcal{KF}_2^0$.

In order to prove this observation we define the following subsets of $\mathcal{KF}_n^0$:

\begin{center} $K=\{k_j:1\leq j\leq n\}$\\ $I=\{i_j:1\leq j\leq n\}$\\  $F=\{f_j:1\leq j\leq n\}$
\end{center}
\vspace{.3cm}

We also allow the formation of product sets in $\mathcal{KF}_n^0$ in the usual way, so we may write, for example, $KFI=\{kfi:k\in K, i\in I, f\in F\}$.  So if $n=2$, we could explicitly write

\begin{center} $KFI=\{f_1i_1, f_1i_2, k_2f_1i_1, k_2f_1i_2, f_2i_*\}$.
\end{center}
\vspace{.3cm}

We will now adopt a notational convention which will not lead to ambiguity in the context of this paper, and which will help us clearly delineate word types in $\mathcal{KF}_n^0$.  Suppose $E$ is a set which is the $n$-times product of the sets $K$, $I$, and $F$ (in any order).  Then we denote by $(E)_r$ the set of all \textit{reduced} words $\omega\in E$, i.e. those which do not admit any representation as a word of length $<n$.  So, under this convention, if $n=2$ we would have

\begin{center} $(KFI)_r=\{k_2f_1i_1, k_2f_1i_2\}$.
\end{center}
\vspace{.3cm}

We are now ready to prove our main Theorem \ref{thm_main}, which is a consequence of the more detailed theorem below.

\begin{thm}  Let $(X,\tau_1,...,\tau_n)$ be a saturated $n$-topological space.  Then $\mathcal{KF}_n^0$ is contained in the union of the sets in the left-hand column of the table below.  The number of distinct elements in each such set is at most as listed in the right-hand column.\\

\begin{minipage}{0.5\textwidth}

\renewcommand{\arraystretch}{1.25}
\begin{tabular}{|c|r|}
\hline
Word-Type & Number of Words\\
\hline
$\{\Id\}$ & $1$\\
$IFK=\{0\}$ & $1$\\
$I$ & $n$\\
$K$ & $n$\\
$IK$ & $n$\\
$KI$ & $n$\\
$IKI$ & $n$\\
$KIK$ & $n$\\
$F$ & $n$\\
$IF$ & $n$\\
$FF$ & $n^2$\\
$FI$ & $n+\binom{n}{2}$\\
$FK$ & $n + \binom{n}{2}$\\
$FIF$ & $n+\binom{n}{2}$\\
$KIF$ & $n$\\
$FIK$ & $n+\binom{n}{2}$\\\hline
\end{tabular}

\end{minipage} \begin{minipage}{0.5\textwidth}

\renewcommand{\arraystretch}{1.25}
\begin{tabular}{|c|r|}
\hline
Word-Type & Number of Words\\
\hline
$FKI$ & $n+\binom{n}{2}$\\
$(KF)_r$ & $\binom{n}{2}$\\
$(KFK)_r$ & $2\binom{n+1}{3}$\\
$(KFI)_r$ & $2\binom{n+1}{3}$\\
$(KFF)_r$ & $ \binom{n}{2}\cdot n$\\
$(FKF)_r$ & $\binom{n}{2}+2\binom{n}{3}$\\
$(FIKI)_r$ & $\binom{n}{2}$\\
$(FKIK)_r$ & $\binom{n}{2}$\\
$(FKIF)_r$ & $\binom{n}{2}$\\
$(KFIK)_r$ & $2\binom{n+1}{3}$\\
$(KFKI)_r$ & $2\binom{n+1}{3}$\\
$(KFIF)_r$ & $2\binom{n+1}{3}$\\
$(KFKF)_r$ & $\binom{n}{2}+5\binom{n}{3}+5\binom{n}{4}$\\
$(KFIKI)_r$ & $\binom{n}{2}+2\binom{n}{3}$\\
$(KFKIK)_r$ & $\binom{n}{2}+2\binom{n}{3}$\\
$(KFKIF)_r$ & $\binom{n}{2}+2\binom{n}{3}$\\
\hline
\end{tabular}

\end{minipage}
\vspace{.3cm}

Consequently, the number of elements of $\mathcal{KF}_n^0$ is at most

\begin{align*}
p(n) &= 5\binom{n}{4}+10\binom{n+1}{3}+13\binom{n}{3}+(n+14)\binom{n}{2}+n^2+14n+2\\
&= \frac{5}{24}n^4+\frac{37}{12}n^3+\frac{79}{24}n^2+\frac{101}{12}n+2
\end{align*}
\vspace{.3cm}

\noindent and the number of elements of $\mathcal{KF}_n$ is at most $2p(n)$.
\end{thm}

\begin{proof}  Let $\mathcal{X}$ be the union of all of the sets in the table above, so we want to prove $\mathcal{KF}_n^0\subseteq\mathcal{X}$.  For this, it suffices to check that \textbf{(A)} for each set $E$ listed in the table above, and for each $1\leq x\leq n$, we have $k_xE, i_xE, f_xE\subseteq\mathcal{X}$.  Our second goal \textbf{(B)} is to establish the listed upper bound for the cardinality of each set.

We can begin the verification by making these observations:

\begin{itemize}
		\item Every $0$- and $1$-letter word type in $\mathcal{KF}_n^0$ (i.e. the elements of $\{\Id\}$, $K$, $I$, and $F$) is accounted for in the table.
		\item There are $3^2=9$ possible $2$-letter word types.  By Lemma \ref{lemma_basics_unsaturated}, we have $II=I$ and $KK=K$, and the other seven possible types are accounted for on the table.  So all elements of $\mathcal{KF}_n^0$ which admit a word representation of length $\leq2$ are contained in $\mathcal{X}$.
		\item There are $3^3=27$ possible $3$-letter word types.  Ten of these reduce to $2$-letter words using $II=I$ and $KK=K$, which by the previous bullet point, are already accounted for in the table.  At most seventeen types remain, and among these, we know that $IFK=IFI=0$ by Lemma \ref{lemma_zero}, while $FFK=KFK$, $FFI=KFI$, and $FFF=KFF$ by Lemma \ref{lemma_ffk}.  Also $IKF=IF$ by Lemmas \ref{lemma_banakh} and \ref{lemma_basics_unsaturated}, and since $F\subseteq KF$, we have $IFF\subseteq IFKF\subseteq \{0\}F\subseteq\{0\}$.  This leaves eleven other possible $3$-letter word types, each of which is listed in the table.  Therefore, all elements of $\mathcal{KF}_n^0$ which admit a word representation of length $\leq 3$ are already contained in a subset listed in the table.
\end{itemize}
\vspace{.3cm}

By the last bullet point above, we see that whenever $E$ consists of $\leq2$-letter words, then indeed we have $k_xE, i_xE, f_xE\subseteq\mathcal{X}$ for each $1\leq x\leq n$, which establishes \textbf{(A)} for the sets $\{\Id\}$, $K$, $I$, $IK$, $KI$, $F$, $IF$ $FF$, $FI$, $FK$, and $(KF)_r$.  \textbf{(A)} is also immediate for the set $\{0\}$.

The cardinality bounds \textbf{(B)} are immediate for the sets $\{\Id\}$, $\{0\}$, $K$, $I$, $IK$, $KI$, $F$, and $FF$.  By Lemma \ref{lemma_if} the set $IF$ consists of words of the form $i_xf_*$ ($1\leq x\leq n$), of which there are $n$ many.  The set $(KF)_r$ consists of elements of the form $k_xf_y$ which do not reduce to $1$-letter representations; by Lemma \ref{lemma_basics_unsaturated}, it is necessary that $x>y$.  There are $\binom{n}{2}$ many pairs $(x,y)$ with $x>y$, so $\#(KF)_r\leq\binom{n}{2}$.  Lastly, by Lemma \ref{lemma_fk}, the set $FK$ consists of words of the form $f_xk_y$ where $1\leq x\leq y\leq n$; there are $n+\binom{n}{2}$ many such pairs $(x,y)$, and thus $FK$ consists of no more than $n+\binom{n}{2}$ elements.  A similar argument yields the same number for $FI$.

So to finish the proof, it remains only to check \textbf{(A)} and \textbf{(B)} for those sets $E$ which consist of words of length $\geq 3$.\\

\underline{The sets $IKI$ and $KIK$.}  By Lemma \ref{lemma_banakh}, every element of $IKI$ has the form $i_yk_*i_*$ for some $1\leq y\leq n$, and thus $\#IKI\leq n$, establishing \textbf{(B)}.  For any $1\leq x\leq n$, we have $k_xi_yk_*i_*=k_xi_*k_*i_*=k_xi_*$ by Lemma \ref{lemma_basics_unsaturated}, so $k_xIKI\subseteq KI\subseteq\mathcal{X}$.  Also $i_xi_yk_*i_*=i_{\max(x,y)}k_*i_*$ by Lemma \ref{lemma_basics_unsaturated}, so $i_xIKI\subseteq IKI\subseteq\mathcal{X}$.  The word $f_xi_yk_*i_*$ either reduces to a $\leq3$-letter word, in which case it is a member of $\mathcal{X}$ by our previous remarks; or it does not reduce, in which case $f_xi_yk_*i_*\in (FIKI)_r\subseteq\mathcal{X}$.  This establishes \textbf{(A)}, and the arguments are similar for $KIK$.\\

\underline{The set $FIF$.}  By Lemma \ref{lemma_if}, every element of $FIF$ has the form $f_yi_zf_*$, so $FIF\subseteq FIf_1$.  Therefore \textbf{(B)} $\#FIF\leq \#FI\leq n+\binom{n}{2}$.  For $1\leq x\leq n$, we have either $k_xf_yi_zf_*\in (KFIF)_r$ or $k_xf_yi_zf_*$ reduces to a shorter word; in either case we obtain $k_xf_yi_zf_*\in\mathcal{X}$ and hence $k_xFIF\subseteq\mathcal{X}$.  By Lemma \ref{lemma_zero} we see $i_xf_yi_zf_*=0f_*=0\in\mathcal{X}$, and by Lemma \ref{lemma_ffk} we see $f_xf_yi_zf_*=k_xf_yi_zf_*\in\mathcal{X}$, establishing \textbf{(A)}.\\

\underline{The set $KIF$.}  By Lemmas \ref{lemma_banakh} and \ref{lemma_if}, every element of $KIF$ has the form $k_yi_*f_*$, where $1\leq y\leq n$, so \textbf{(B)} holds.  For any $1\leq x\leq n$, $k_xk_yi_*f_*=k_{\max(x,y)}i_*f_*\in KIF\subseteq\mathcal{X}$, and $i_xk_yi_*f_*=i_xk_*i_*k_*f_*=i_xk_*f_*\in IKF\subseteq\mathcal{X}$.  The word $f_xk_yi_*f_*$ either reduces to a $\leq 3$-letter word or else lies in $(FKIF)_r$; in either case it lies in $\mathcal{X}$, establishing \textbf{(A)}.\\

\underline{The sets $FIK$ and $FKI$.}  Elements of $FIK$ have the form $f_yi_zk_*$, so $FIK=FIk_1$, and \textbf{(B)} $\#FIK\leq\#FI\leq n+\binom{n}{2}$.  For \textbf{(A)}, note that for any $x$, the word $k_xf_yi_zk_*=f_xf_yi_zk_*$ either reduces to a $\leq3$ letter word or else lies in $(KFIK)_r$, so it lies in $\mathcal{X}$, while $i_xFIK=\{0\}K=\{0\}\subseteq\mathcal{X}$ as well.  The arguments are similar for $FKI$.\\

\underline{The sets $(KFK)_r$ and $(KFI)_r$.}  Elements of $(KFK)_r$ have the form $k_yf_zk_w$.  To establish \textbf{(A)}, we note that for $1\leq x\leq n$, we have $k_xk_yf_zk_w=k_{\max(x,y)}f_zk_w$ by Lemma \ref{lemma_basics_unsaturated}, $i_xk_yf_zk_w=i_xk_*f_zk_w=i_xf_zk_w$ by Lemma \ref{lemma_banakh}, and $f_xk_yf_zk_w=k_{\max(x,y)}f_zk_w$ by Lemma \ref{lemma_fkfk}.  Then all three words admit representations of length $\leq 3$, and therefore lie in $\mathcal{X}$.

For \textbf{(B)}, since $k_yf_zk_w$ cannot be written with $\leq 2$ letters, by Lemma \ref{lemma_basics_unsaturated} it is necessary that $y>z$.  Also, by Lemma \ref{lemma_fk}, we may assume that $w\geq z$.  The number of triples $(y,z,w)$ with $y>z$ and $z\leq w$ may be found by the following reasoning: either $z=w$ or $z\neq w$.  If $z=w$, we find $\binom{n}{2}$ many triples $(y,z,z)$ with $y>z$.  If $z\neq w$, either $w=y$ or $w\neq y$.  If $w=y$ we again obtain $\binom{n}{2}$ many triples $(y,z,y)$.  If $w\neq y$, then there are $\binom{n}{3}$ many sets of distinct numbers $\{y,z,w\}$ where $z$ is minimal; these each yield two choices of ordered triples $(y,z,w)$ or $(w,z,y)$.  So the cardinality of $(KFK)_r$ is no more than $\binom{n}{2}+\binom{n}{2}+2\binom{n}{3}= 2\binom{n+1}{3}$.  The arguments are similar for $(KFI)_r$.\\

\underline{The set $(KFF)_r$.}  Elements of $(KFF)_r$ have the form $k_yf_zf_w$, which can be rewritten as $k_yf_zk_wf_w$; thus the arguments to establish \textbf{(A)} are exactly analogous to those given for the case of $(KFK)_r$.  For \textbf{(B)}, we note that since $k_yf_zf_w$ cannot be written as a word of length $\leq 2$, it must be the case that $k_yf_z\in(KF)_r$.  Therefore $\#(KFF)_r\leq \#(KF)_r\cdot\#F = \binom{n}{2}\cdot n$.\\

\underline{The set $(FKF)_r$.}  Elements of $(FKF)_r$ have the form $f_yk_zf_w$.  To establish \textbf{(A)}, note that for $1\leq x\leq n$, we have $k_xf_yk_zf_w=f_xf_yk_zf_w$ by Lemma \ref{lemma_ffk}, and this word either admits a word representation of length $\leq 3$ and therefore lies in $\mathcal{X}$, or else it lies in $(KFKF)_r\subseteq\mathcal{X}$.  Also $i_xf_yk_zf_w=0f_w=0\in\mathcal{X}$.

For \textbf{(B)}, since $f_yk_zf_w$ cannot be written with $\leq 2$ letters, by Lemma \ref{lemma_fkf} it is necessary that $z>y$ and $z>w$.  We have either $y=w$ or $y\neq w$.  If $y=w$ we are looking for triples of the form $(y,z,y)$ with $z>y$, of which there $\binom{n}{2}$ many.  If $y\neq w$, we find $\binom{n}{3}$ many sets $\{y,z,w\}$ of distinct numbers where $z$ is maximal; each of these yields two choices of ordered triples $(y,z,w)$ or $(w,z,y)$.  So the cardinality of $(FKF)_r$ is no more than $\binom{n}{2}+2\binom{n}{3}$.\\

At this point, we pause to observe the following: combining all the arguments in the previous parts, we have shown that if $o\in\mathcal{KF}_n^0$ admits any representation as a word of length $\leq 3$, then for every $1\leq x\leq n$, we have $k_xo, i_xo, f_xo\in\mathcal{X}$.  All words of length $\leq 4$ have this form, so put in other words, we have now shown:\\

\begin{itemize}
		\item All elements of $\mathcal{KF}_n^0$ which admit a word representation of length $\leq 4$ are already contained in a subset listed in the table. 
\end{itemize}
\vspace{.3cm}

\underline{The sets $(FIKI)_r$, $(FKIK)_r$, and $(FKIF)_r$.}  Elements of $(FIKI)_r$ have the form $f_yi_zk_*i_*$, where $y<z$ by Lemma \ref{lemma_fiki}.  There are $\binom{n}{2}$ many such pairs $(y,z)$, so \textbf{(B)} $\#(FIKI)_r\leq \binom{n}{2}$.  For \textbf{(A)}, note that for any $x$, the word $k_xf_yi_zk_*i_*=f_xf_yi_zk_*i_*$ either reduces to a $\leq4$ letter word or else lies in $(KFIKI)_r$, so it lies in $\mathcal{X}$; while $i_xFIKI=\{0\}KI=\{0\}\subseteq\mathcal{X}$ as well.  The arguments are similar for $(FKIK)_r$ and $(FKIF)_r$.\\

\underline{The sets $(KFIK)_r$, $(KFKI)_r$, and $(KFIF)_r$.}  All elements of $(KFIK)_r$ have the form $k_yf_zi_wk_*$ where $y>z$ and $z\leq w$, which implies $(KFIK)_r\subseteq (KFI)_rk_1$ and therefore $\#(KFIK)_r\leq\#(KFI)_r\leq2\binom{n+1}{3}$, establishing \textbf{(B)}.  For $1\leq x\leq n$, we have $k_xk_yf_zi_wk_*=k_{\max(x,y)}f_zi_wk_*$ by Lemma \ref{lemma_basics_unsaturated}, and $i_xk_yf_zi_wk_*=i_xk_*f_zi_wk_*=i_xf_zi_wk_*$ by Lemma \ref{lemma_banakh}, and $f_xk_yf_zi_wk_*=k_{\max(x,y)}f_zi_wk_*$ by Lemma \ref{lemma_fkfk}.  Each of these words has a representation of length $\leq 4$, and therefore lies in $\mathcal{X}$, establishing \textbf{(A)}.  The arguments are similar for $(KFKI)_r$ and $(KFIF)_r$.\\

\underline{The set $(KFKF)_r$.}  For $1\leq x\leq n$, we have $k_xKFKF\subseteq KFKF\subseteq\mathcal{X}$ by Lemma \ref{lemma_basics_unsaturated}, $i_xKFKF\subseteq IFKF\subseteq\mathcal{X}$ by Lemma \ref{lemma_banakh}, and $f_xKFKF\subseteq KFKF\subseteq\mathcal{X}$ by Lemma \ref{lemma_fkfk}, so \textbf{(A)} holds.

To establish \textbf{(B)}, we observe that every element of $(KFKF)_r$ has the form $k_xf_yk_zf_w$, and because this cannot be shortened to a word of length $\leq 3$, we must have $x>y$ by Lemma \ref{lemma_basics_unsaturated}, and $y<z$, $z>w$ by Lemma \ref{lemma_fkf}.  So we are looking for ordered quadruples $(x,y,z,w)$ which alternate in magnitude with $x>y$, $y<z$, $z>w$.  There are $\binom{n}{2}$ many such quadruples if $x=z$ and $y=w$; there are $2\binom{n}{3}$ many if $x=z$ but $y\neq w$; and there are $2\binom{n}{3}$ many if $y=w$ but $x\neq z$.  If $x=w$, then necessarily $y<x$ and $z>x$, which yields an additional $\binom{n}{3}$ possible quadruples.  If all of $x,y,z,w$ are distinct, then either $x$ or $z$ is maximal.  If $x$ is maximal then the choice of minimality for $y$ or $w$ determines the quadruple, yielding $2\binom{n}{4}$ quadruples.  If $z$ is maximal then either $y$ or $w$ is minimal; if $w$ is minimal the quadruple is determined, whereas if $y$ is minimal then there are $2$ ways to assign $x$ and $w$.  This gives another $(1+2)\binom{n}{4}$ quadruples where $z$ is maximal.  Thus we compute a bound of $\#(KFKF)_r\leq \binom{n}{2}+(2+2+1)\binom{n}{3}+(2+1+2)\binom{n}{4}$, as in the table.\\

At this point, our computations up to this point have shown:

\begin{itemize}
		\item All elements of $\mathcal{KF}_n^0$ which admit a word representation of length $\leq 5$ are already contained in a subset listed in the table. 
\end{itemize}
\vspace{.3cm}

\underline{The sets $(KFIKI)_r$, $(KFKIK)_r$, and $(KFKIF)_r$.}  Every element of $(KFIKI)_r$ has the form $k_yf_zi_wk_*i_*$, and we check that for any $1\leq x\leq n$, we have $k_xk_yf_zi_wk_*i_*=f_xk_yf_zi_wk_*i_*=k_{\max(x,y)}f_zi_wk_*i_*$ by Lemmas \ref{lemma_basics_unsaturated} and \ref{lemma_fkfk}, while $i_xk_yf_zi_wk_*i_*=i_xk_*f_zi_wk_*i_*=i_xf_zi_wk_*i_*$ by Lemma \ref{lemma_banakh}.  In all three cases we find representations of length $\leq 5$, so $k_x(KFIKI)_r$, $i_x(KFIKI)_r$, $f_x(KFIKI)_r\subseteq\mathcal{X}$ and we have proven \textbf{(A)}.

For \textbf{(B)}, we note that since $k_yf_zi_wk_*i_*$ does not reduce to a word of length $\leq4$, we must have $y>z$ by Lemma \ref{lemma_basics_unsaturated}, and by Lemma \ref{lemma_fiki} we have $w>z$.  Thus we are looking for triples $(y,z,w)$ with $y>z$ and $z<w$.  By arguments analogous to those in the case of $(FKF)_r$, we compute that $\#(KFIKI)_r\leq \binom{n}{2}+2\binom{n}{3}$.  The arguments for $(KFKIK)_r$ and $(KFKIF)_r$ are similar.\\

This completes the proof.
\end{proof}

\begin{example}[Separating KFKF Words]  In \cite{Banakh_2018a}, the authors show that $\#\mathcal{K}_n\leq 12n+2$ for a saturated $n$-topological space, so we expect the size of the Kuratowski monoid to grow linearly with $n$.  Our corresponding formula $p(n)$ in Theorem \ref{thm_main} implies quartic growth for the Kuratowski-Gaida-Eremenko monoid $\mathcal{KF}_n$.  As is evident from the proof, the sole reason for this is that the set of reduced words $(KFKF)_r=\{k_xf_yk_zf_w:x> y,y< z,z> w,1\leq x,y,z,w\leq n\}$ is expected to contain $\binom{n}{2}+5\binom{n}{3}+5\binom{n}{4}$ elements.

It is interesting to see a natural example of a saturated $4$-topological space in which the elements of $(KFKF)_r$ are distinct.  Consider $(\mathbb{R}^3,\tau_1,\tau_2,\tau_3,\tau_4)$, where $\tau_1=\tau_s\times\tau_s\times\tau_s$, $\tau_2=\tau_s\times\tau_s\times\tau_u$, $\tau_3=\tau_s\times\tau_u\times\tau_u$, and $\tau_4=\tau_u\times\tau_u\times\tau_u$.  Define $B=((1,2)\times(0,2)\times(0,2))\cup((0,2)\times(1,2)\times(0,2))\cup((0,2)\times(0,2)\times(1,2))$, and let $\{C_n:n\in\mathbb{N}\}$ be a countably infinite collection of pairwise disjoint $\tau_4$-closed sub-cubes of $(0,1)\times(0,1)\times(0,1)$ with the property that if $C=\bigcup_{n\in\mathbb{N}}C_n$, then the set of $\tau_4$-derived points of $C$ is exactly $C'=k_4C\backslash C=(\{1\}\times[0,1]\times[0,1])\cup([0,1]\times\{1\}\times[0,1])\cup([0,1]\times[0,1]\times\{1\})$.  We denote $B_{\mathbb{Q}}=B\cap(\mathbb{Q}\times\mathbb{Q}\times\mathbb{Q})$, and we take for our initial set $A=B_{\mathbb{Q}}\cup \left(\bigcup_{n\in\mathbb{N}}i_4C\right)$.

We also consider the particular open cube $i_4C_0\subseteq A$, say $i_4C_0=(x_0,x_1)\times(y_0,y_1)\times(z_0,z_1)$, and we label the following sets:

\begin{align*}
\phi &= (x_0,x_1)\times(y_0,y_1)\times\{z_1\} &= \text{the upper face of } C_0;\\
\psi &= (x_0,x_1)\times\{y_1\}\times(z_0,z_1) &= \text{the forward face of } C_0;\\
q &= (0,1)\times\{1\}\times(0,1) &= \text{the inner rear face of } B;\\
r &= \{1\}\times(0,1)\times(0,1) &= \text{the inner left face of } B;\\
Q &= (0,2)\times\{2\}\times(0,2) &= \text{the forward face of } B;\\
R &= \{2\}\times(0,2)\times(0,2) &= \text{the right face of } B;\\
U &= [((1,2)\times(0,2))\cup((0,2)\times(1,2))]\times\{0\} &= \text{the outer lower face of } B;\\
V &= ((1,2)\times\{0\}\times(0,2))\cup((0,2)\times\{0\}\times(1,2)) &= \text{the outer rear face of } B.
\end{align*}

\begin{figure}[b] \caption{From left to right: the set $C_0$ and its faces; the set $B$ and its faces; typical basic open neighborhoods in $\tau_1$, $\tau_2$, $\tau_3$.} \includegraphics[scale=0.28]{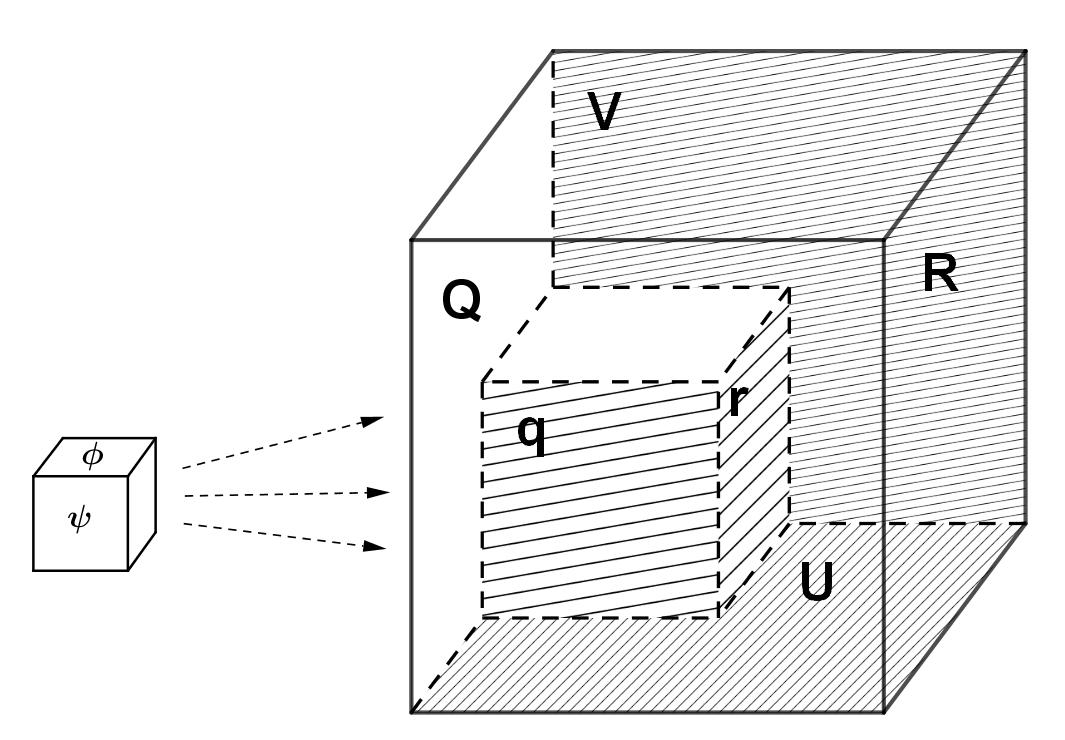} \includegraphics[scale=0.1]{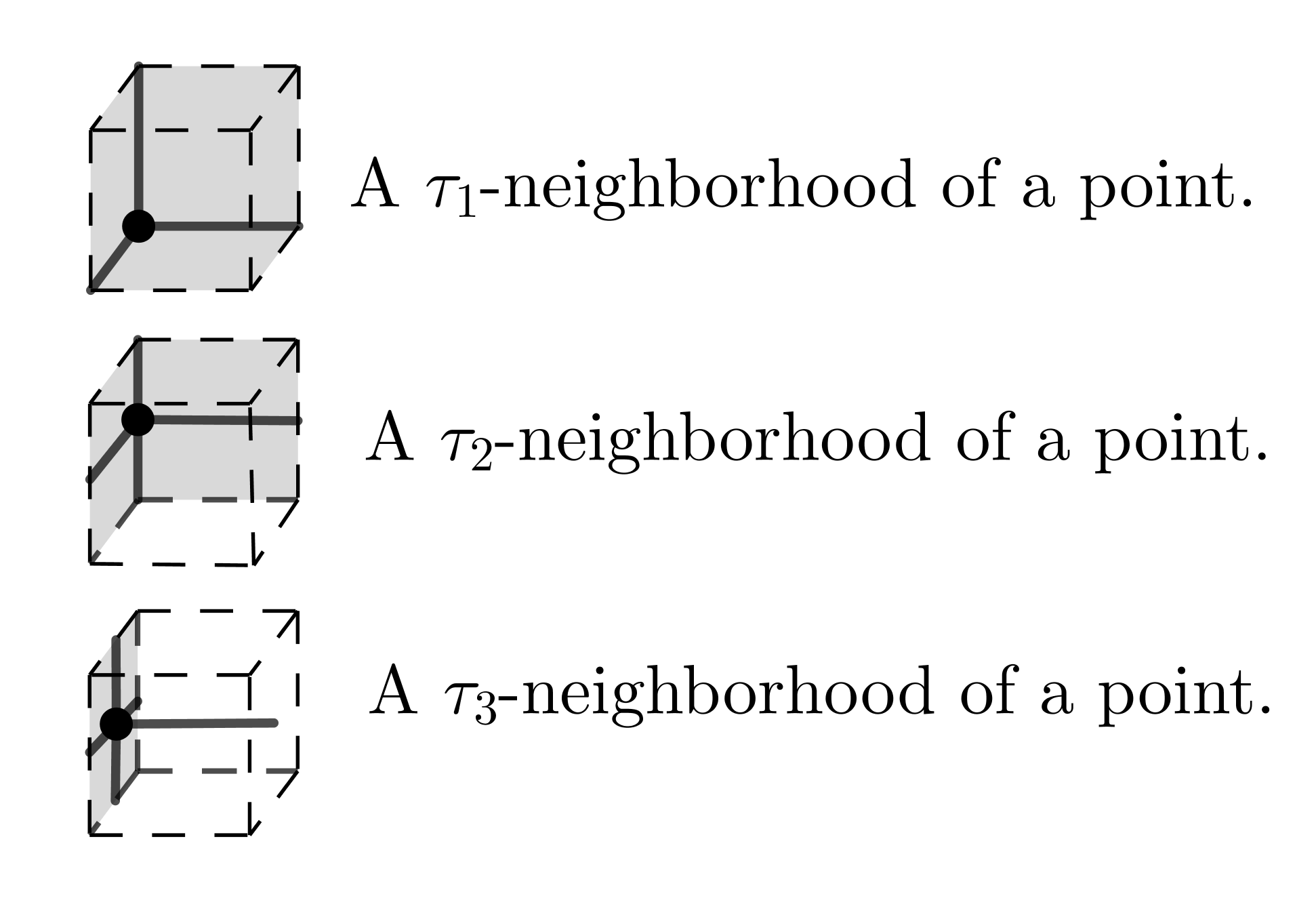}
\end{figure}

Then by direct computation, one may verify the following properties about the sets $k_xf_yk_zf_wA$, which differentiate all possible ordered quadruples $(x,y,z,w)$ satisfying $x>y,y<z,z>w$:

\begin{enumerate}
		\item \begin{enumerate}
				\item If $w=1$ then $\phi,\psi$ are disjoint from $k_xf_yk_zf_wA$.
				
				\item If $w=2$ then $\phi\subseteq k_xf_yk_zf_wA$ but $\psi\cap k_xf_yk_zf_wA=\emptyset$.
				
				\item If $w=3$ then $\phi,\psi\subseteq k_xf_yk_zf_wA$.
		\end{enumerate}
		
		\item \begin{enumerate}
				\item If $z=2$ then $Q,R$ are disjoint from $k_xf_yk_zf_wA$.
				
				\item If $z=3$ then $Q\subseteq k_xf_yk_zf_wA$ but $R\cap k_xf_yk_zf_wA=\emptyset$.
				
				\item If $z=4$ then $Q,R\subseteq k_xf_yk_zf_wA$.
		\end{enumerate}
		
		\item \begin{enumerate}
				\item If $y=1$ then $U,V$ are disjoint from $k_xf_yk_zf_wA$.
				
				\item If $y=2$ then $U\subseteq k_xf_yk_zf_wA$ but $V\cap k_xf_yk_zf_w=\emptyset$.
				
				\item If $y=3$ then $U,V\subseteq k_xf_yk_zf_wA$.
		\end{enumerate}
		
		\item \begin{enumerate}
				\item If $x=2$ then $q,r$ are disjoint from $k_xf_yk_zf_wA$.
				
				\item If $x=3$ then $q\subseteq k_xf_yk_zf_wA$ but $r\cap k_xf_yk_zf_wA=\emptyset$.
				
				\item If $x=4$ then $q,r\subseteq k_xf_yk_zf_wA$.
		\end{enumerate}
\end{enumerate}
\vspace{.3cm}

From the above, distinct quadruples $(x,y,z,w)$ yield distinct sets $k_xf_yk_zf_wA$, and therefore

\begin{center} $\#(KFKF)_r[A]=\#\{k_xf_yk_zf_wA: 1\leq x,y,z\leq n,x>y,y<z,z>w\}=\binom{4}{2}+5\binom{4}{3}+5\binom{4}{2}=31$.
\end{center}
\end{example}

\section{Separating Kuratowski-Gaida-Eremenko Words}

The goal of this section is to prove that our upper bound $p(n)$ is sharp for every $n$.  Guided by the results of the previous section, we introduce the following definition: a word in the generators $\{k_x,i_x,f_x:1\leq x\leq n\}$ (formally, an element of the free semigroup on $3n$ letters) will be called a \textbf{Kuratowski-Gaida-Eremenko word}, or KGE-word, if it has one of the following forms:\\

\begin{itemize}
		\item $\Id$ or $0=i_*f_*k_*$,
		\item $k_x$, $i_x$, $i_xk_*$, $k_xi_*$, $i_xk_*i_*$, $k_xi_*k_*$, $f_x$, $i_xf_*$, or $k_xi_*f_*$,
		\item $f_xf_y$,
		\item $k_xf_y$ where $x>y$, 
		\item $f_xi_yk_*i_*$, $f_xk_yi_*k_*$, or $f_xk_yi_*f_*$ where $x<y$,
		\item $f_xi_y$, $f_xk_y$, $f_xi_yf_*$, $f_xi_yk_*$, or $f_xk_yi_*$ where $x\leq y$,
		\item $k_xf_yf_z$ where $x>y$,
		\item $f_xk_yf_z$ where $x<y$ and $y<z$,
		\item $k_xf_yk_z$, $k_xf_yi_z$, $k_xf_yi_zk_*$, $k_xf_yk_zi_*$, or $k_xf_yi_zf_*$ where $x>y$ and $y\leq z$,
		\item $k_xf_yi_zk_*i_*$, $k_xf_yk_zi_*k_*$, or $k_xf_yk_zi_*f_*$ where $x>y$ and $y<z$,
		\item $k_xf_yk_zf_w$ where $x>y$, $y<z$, and $z>w$.
\end{itemize}
\vspace{.3cm}

We understand the $*$-notation as imposing an equivalence relation on the KGE-words: for example, although strictly speaking $i_1f_1k_1$ and $i_1f_1k_2$ are distinct words in the free semigroup, we regard them here as merely two representations of the same KGE-word $0$; on the other hand $f_1f_1$ and $f_1f_2$ are distinct KGE-words.    With this understanding in place, the number of KGE-words is $p(n)$.  For convenience, we allow words and operators to be used interchangeably when the precise meaning is clear.  So each KGE-word corresponds to at most one element of $\mathcal{KF}_n$, whereas \textit{a priori} an element of $\mathcal{KF}_n$ may be represented by more than one KGE-word. 

We note that in any monoid $\mathcal{KF}_n^0$, by Lemmas \ref{lemma_basics_unsaturated} through \ref{lemma_fkfk}, we have the following set inclusions:

\begin{center} $KF\supseteq (KF)_r\cup F$;\\ $KFK\supseteq (KFK)_r\cup FK$;\\ $KFI\supseteq (KFI)_r\cup FI$;\\ $KFIF\supseteq (KFIF)_r\cup FIF$;\\ $KFKIF\supseteq (KFKIF)_r\cup (FKIF)_r$;\\ $KFIK\supseteq (KFIK)_r\cup FIK$;\\ $KFKIK\supseteq (KFKIK)_r\cup (FKIK)_r$;\\ $KFKI\supseteq (KFKI)_r\cup FKI$;\\ $KFIKI\supseteq (KFIKI)_r\cup (FIKI)_r$;\\ $KFKF\supseteq (KFKF)_r\cup (KFF)_r\cup (FKF)_r\cup FF$.
\end{center}
\vspace{.3cm}

Therefore the following holds.

\begin{prop} \label{prop_cases}  Each KGE-word belongs to at least one of the following sets in $\mathcal{KF}_n^0$:\\

\begin{itemize}
		\item $\{\Id\}$ or $\{0\}$;
		\item $K$, $I$, $IK$, $KI$, $KIK$, $IKI$, $IF$, $KIF$;
		\item $KF$, $KFK$, $KFI$;
		\item $KFKF$; or
		\item $KFIF\cup KFKIF$, $KFIK\cup KFKIK$, $KFKI\cup KFIKI$.
\end{itemize}
\end{prop}

For the reader's convenience, we note that the $17$ sets above correspond to the $17$ distinct even operators which comprise the monoid $\mathcal{KF}_1^0$ from \cite{Gaida_Eremenko_1974a}.\\

\begin{thm:thm2}  For every $n\geq 1$, there exists a saturated polytopological space $(X,\tau_1,...,\tau_n)$ in which $\#\mathcal{KF}_n^0=p(n)$ and $\#\mathcal{KF}_n=2p(n)$.  In fact, there is an initial set $A\subseteq X$ such that $\#\{oA:o\in\mathcal{KF}_n\}=2p(n)$.
\end{thm:thm2}

\begin{proof}  Applying Lemma \ref{lemma_disjoint_union}, it suffices to demonstrate the following:  For any pair of distinct KGE-words $\omega_1,\omega_2\in\mathcal{KF}_n$, there exists a saturated $n$-topological space $X^{\omega_1,\omega_2}$ and a subset $A^{\omega_1,\omega_2}\subseteq X^{\omega_1,\omega_2}$ in which $\omega_1A^{\omega_1,\omega_2}\neq \omega_2A^{\omega_1,\omega_2}$.  We verify the claim for $\omega_1\neq\omega_2$ by using the cases delineated in Proposition \ref{prop_cases}.\\

\noindent \underline{Case 1: $\omega_1\in E_1$ and $\omega_2\in E_2$, where $E_1$ and $E_2$ are distinct subsets from Proposition \ref{prop_cases}.}  Then we may take for our separating space $(\mathbb{R},\tau_1,...,\tau_n)$ where $\tau_1=...=\tau_n=\tau_u$, and take for our initial set $A$ the example exhibited by Gaida-Eremenko in \cite{Gaida_Eremenko_1974a}.  In this case, because all topologies are equal, the monoid $\mathcal{KF}_n^0$ is actually equal to $\mathcal{KF}_1^0$ and we get the following reductions: $KFIF\cup KFKIF=FIF$, $KFIK\cup KFKIK=FIK$, $KFKI\cup KFIKI=FKI$, $KFKF=FF$, $KF=F$, $KFK=FK$, $KFI=FI$.  But elements $\omega_1,\omega_2$ taken from distinct word types will produce different sets $\omega_1A\neq\omega_2A$, as demonstrated by Gaida and Eremenko.\\

\noindent \underline{Case 2: $\omega_1,\omega_2\in E$ where $E=K$, $I$, $IK$, $KI$, $KIK$, $IKI$, $IF$, or $KIF$.}  We assume, for example, that $\omega_1,\omega_2\in KIK$.  We have $\omega_1=k_{x_1}i_*k_*$ and $\omega_2=k_{x_2} i_*k_*$ where $1\leq x_1,x_2\leq n$, and since $\omega_1\neq\omega_2$, we have $x_1\neq x_2$.  Assume without loss of generality that $x_1<x_2$, and take for a separating space $(\mathbb{R},\tau_1,...,\tau_n)$ where $\tau_1=...=\tau_{x_1}=\tau_s$ and $\tau_{x_1+1}=...=\tau_n=\tau_u$.  Take the initial set $A$ from Example \ref{ex_usual_sorg}.  Then $\omega_1A=k_1i_*k_*A\neq k_ni_*k_*A=\omega_2A$.  The proofs for the other sets $E=K,I,...$ etc. are similar because words in these sets $E$ depend on only one index, and we leave them to the reader.\\

\noindent \underline{Case 3: $\omega_1,\omega_2\in KF$.}  If $\omega_1,\omega_2\in KF$, then we have $\omega_1=k_{x_1}f_{y_1}$ and $\omega_2=k_{x_2}f_{y_2}$ where $1\leq x_1,y_1,x_2,y_2\leq n$, $x_1\geq y_1$, and $x_2\geq y_2$.  Assuming $(x_1,y_1)\neq(x_2,y_2)$, we have either $x_1\neq x_2$ or $y_1\neq y_2$.

\indent\indent \textit{Sub-Case (a):}  Suppose $y_1\neq y_2$; without loss of generality assume $y_1<y_2$.  Then take for a separating space $(\mathbb{R},\tau_1,..,\tau_n)$ where $\tau_1=...=\tau_{y_1}=\tau_s$ and $\tau_{y_1+1}=...=\tau_n=\tau_u$, and take for an initial set $A$ as in Example \ref{ex_usual_sorg}.  Then we have $\omega_1=k_{x_1}f_{y_1}=k_{x_1}f_1$, which is equal to either $f_1$ or $k_nf_1$ depending on the value of $x_1$.  On the other hand since $x_2\geq y_2>y_1$, we have $\omega_2=k_{x_2}f_{y_2}=k_nf_n=f_n$.  Since $f_1A\neq k_nf_1A\neq f_nA$, we conclude $\omega_1A\neq\omega_2A$ as desired.

\indent\indent \textit{Sub-Case (b):}  Suppose $y_1=y_2$ but $x_1\neq x_2$; without loss of generality assume $x_1<x_2$.  Take for a separating space $(\mathbb{R},\tau_1,..,\tau_n)$ where $\tau_1=...=\tau_{x_1}=\tau_s$ and $\tau_{x_1+1}=...=\tau_n=\tau_u$, and take the usual initial set $A$ as in Example \ref{ex_usual_sorg}.  Then since $y_1\leq x_1$, we have $\omega_1=k_{x_1}f_{y_1}=k_1f_1=f_1$, whereas $\omega_2=k_{x_2}f_{y_2}=k_nf_{y_2}\in\{f_n,k_nf_1\}$.  So $\omega_1A\neq\omega_2A$ as in Example \ref{ex_usual_sorg}.\\

\noindent \underline{Case 4: $\omega_1,\omega_2\in KFK$.}  The idea of this proof is the same as in Case 3.  If $\omega_1,\omega_2\in KFK$, then we have $\omega_1=k_{x_1}f_{y_1}k_{z_1}$ and $\omega_2=k_{x_2}f_{y_2}k_{z_2}$ where $1\leq x_1,y_1,z_1,x_2,y_2,z_2\leq n$, $x_1\geq y_1$, $y_1\leq z_1$, $x_2\geq y_2$, $y_2\leq z_2$.  We have $(x_1,y_1,z_1)\neq(x_2,y_2,z_2)$, and therefore $x_1\neq x_2$, $y_1\neq y_2$ or $z_1\neq z_2$.

\indent\indent \textit{Sub-Case (a):}  Suppose $z_1\neq z_2$, so without loss of generality $z_1<z_2$.  Take for a separating space $(\mathbb{R},\tau_1,..,\tau_n)$ where $\tau_1=...=\tau_{z_1}=\tau_s$ and $\tau_{z_1+1}=...=\tau_n=\tau_u$, and take for an initial set $A$ as in Example \ref{ex_usual_sorg}.  Then since $y_1\leq z_1$, we have $\omega_1=k_{x_1}f_{y_1}k_{z_1}=k_{x_1}f_1k_1$, which is equal to either $f_1k_1$ or $k_nf_1k_1$ depending on the value of $x_1$.  On the other hand $\omega_2=k_{x_2}f_{y_2}k_n$, so $\omega_2$ is equal to either $k_1f_1k_n=f_1k_n$, $k_nf_1k_n$, or $k_nf_nk_n=f_nk_n$, depending on the values of $x_2,y_2$.  These five distinct possibilities yield five distinct sets when applied to $A$, so we conclude $\omega_1A\neq\omega_2A$ as desired.

\indent\indent \textit{Sub-Case (b):}  Suppose $z_1=z_2$ but $y_1< y_2$.  Take for a separating space $(\mathbb{R},\tau_1,..,\tau_n)$ where $\tau_1=...=\tau_{y_1}=\tau_s$ and $\tau_{y_1+1}=...=\tau_n=\tau_u$, and take the usual initial set $A$ as in Example \ref{ex_usual_sorg}.  Then, considering all possible values of $x_1,z_1$, we compute that $\omega_1=k_{x_1}f_1k_{z_1} \in \{f_1k_1,f_1k_n,k_nf_1k_1,k_nf_1k_n\}$.  On the other hand since $x_2,z_2\geq y_2>y_1$, we have $\omega_2=k_nf_nk_n=f_nk_n$.  So $\omega_1A\neq\omega_2A$.

\indent\indent \textit{Sub-Case (c):}  Suppose $z_1=z_2$ and $y_1=y_2$ but $x_1<x_2$.  Take for a separating space $(\mathbb{R},\tau_1,..,\tau_n)$ where $\tau_1=...=\tau_{x_1}=\tau_s$ and $\tau_{x_1+1}=...=\tau_n=\tau_u$, and take the usual initial set $A$ as in Example \ref{ex_usual_sorg}.  We compute $\omega_1=k_{x_1}f_{y_1}k_{z_1}=k_1f_1k_{z_1}\in\{f_1k_1,f_1k_n\}$, and $\omega_2=k_{x_2}f_{y_2}k_{z_2}=k_nf_{y_2}k_{z_2}\in\{k_nf_1k_1,k_nf_1k_n,f_nk_n\}$, so $\omega_1A\neq\omega_2A$.\\

\noindent \underline{Case 5: $\omega_1,\omega_2\in KFI$.}  In this case take the same separating space as in Case 4, but for an initial set take $cA$ where $A$ is the initial set from Case 4.  We are done if $\omega_1cA\neq\omega_2cA$, and this follows from Case 4 because both $\omega_1c$ and $\omega_2c$ are elements of $KFK$.  (To verify this, write $\omega_1=k_{x_1}f_{y_1}i_{z_1}$ where $1\leq x_1,y_1,z_1\leq n$, $x_1\geq y_1$, and $y_1\leq z_1$.  Then $\omega_1c=k_{x_1}f_{y_1}ck_{z_1}=k_{x_1}f_{y_1}k_{z_1}\in KFK$, and similarly for $\omega_2$.)\\

\noindent \underline{Case 6: $\omega_1,\omega_2\in KFKF$.}  We proceed similarly to Cases 3 and 4.  We have $\omega_1=k_{x_1}f_{y_1}k_{z_1}f_{w_1}$ and $\omega_2=k_{x_2}f_{y_2}k_{z_2}f_{w_2}$ where $1\leq x_1,y_1,z_1,w_1,x_2,y_2,z_2,w_2\leq n$, $x_1\geq y_1$, $y_1\leq z_1$, $z_1\geq w_1$, $x_2\geq y_2$, $y_2\leq z_2$, and $z_2\geq w_2$.  We also know $(x_1,y_1,z_1,w_1)\neq(x_2,y_2,z_2,w_2)$, which gives us four sub-cases.

\indent\indent \textit{Sub-Case (a):}  Suppose $w_1\neq w_2$, so without loss of generality $w_1<w_2$.  We consider $(\mathbb{R},\tau_1,...,\tau_n)$ with $\tau_1=...=\tau_{w_1}=\tau_s$ and $\tau_{w_1+1}=...=\tau_n=\tau_u$.  Considering all possible values of $x_1,y_1,z_1,x_2,y_2,z_2$, we compute that

\begin{align*}
\omega_1 = k_{x_1}f_{y_1}k_{z_1}f_1 &\in \{f_1f_1, f_nf_1, f_1k_nf_1, k_nf_1f_1, k_nf_1k_nf_1\}\\
\omega_2 = k_{x_2}f_{y_2}k_{z_2}f_n &\in \{f_1f_n, f_nf_n, k_nf_1f_n\}
\end{align*}
\vspace{.3cm}

\noindent from which we conclude $\omega_1A\neq\omega_2A$, where $A$ is the initial set from Example \ref{ex_usual_sorg}.

\indent\indent \textit{Sub-Case (b):}  Suppose $w_1=w_2$ but $z_1<z_2$, and consider $(\mathbb{R},\tau_1,...,\tau_n)$ where $\tau_1=...=\tau_{z_1}=\tau_s$ and $\tau_{z_1+1}=...=\tau_n=\tau_u$.  Since $w_1, y_1\leq z_1$, we get $\omega_1=k_{x_1}f_1k_1f_1\in\{f_1f_1,k_nf_1f_1\}$ whereas $\omega_2=k_{x_2}f_{y_2}k_nf_1\in\{f_1k_nf_1,k_nf_1k_nf_1,f_nf_1\}$, so $\omega_1A\neq\omega_2A$ where $A$ is as in Example \ref{ex_usual_sorg}.

\indent\indent \textit{Sub-Case (c):}  Suppose $w_1=w_2$, $z_1=z_2$ but $y_1<y_2$, and consider $(\mathbb{R},\tau_1,...,\tau_n)$ where $\tau_1=...=\tau_{y_1}=\tau_s$ and $\tau_{y_1+1}=...=\tau_n=\tau_u$.  Since $z_1=z_2\geq y_2$, we get $\omega_1=k_{x_1}f_1k_nf_{w_1}\in\{f_1k_nf_1, k_nf_1k_nf_1, f_1f_n, k_nf_1f_n\}$, whereas since $x_2,z_2\geq y_2$, we have $\omega_2=k_nf_nk_nf_{w_2}\in\{f_nf_1,f_nf_n\}$, so $\omega_1A\neq\omega_2A$ where $A$ is as in Example \ref{ex_usual_sorg}.

\indent\indent \textit{Sub-Case (d):}  Suppose $w_1=w_2$, $z_1=z_2$, $y_1=y_2$ but $x_1<x_2$, and consider $(\mathbb{R},\tau_1,...,\tau_n)$ where $\tau_1=...=\tau_{x_1}=\tau_s$ and $\tau_{x_1+1}=...=\tau_n=\tau_u$.  Since $y_1\leq x_1$, we get $\omega_1=k_1f_1k_{z_1}f_{w_1}\in\{f_1f_1,f_1f_n,f_1k_nf_1\}$ whereas $\omega_2=k_nf_{y_2}k_{z_2}f_{w_2}\in\{k_nf_1f_1, k_nf_1f_n, k_nf_1k_nf_1, f_nf_1\}$, so $\omega_1A\neq\omega_2A$ where $A$ is as in Example \ref{ex_usual_sorg}.
\end{proof}

\noindent \underline{Case 7: $\omega_1,\omega_2\in KFIF\cup KFKIF$.}  We proceed similarly to Cases 3, 4, and 6.  Observe that we may write

\begin{center} $\omega_1=k_{x_1}f_{y_1}\sigma_{z_1}i_1f_*$
\end{center}
\vspace{.3cm}

\noindent with $x_1\geq y_1$, $y_1\leq z_1$, where either $\sigma_{z_1}=i_{z_1}\in I$ (in case $\omega_1\in KFIF)$ or $\sigma_{z_1}=k_{z_1}\in K$ where $z_1>y_1$ (in case $\omega_1\in KFKIF\backslash KFIF$).  Similarly, we may write $\omega_2$ as

\begin{center} $\omega_2=k_{x_2}f_{y_2}\rho_{z_2}i_1f_*$
\end{center}
\vspace{.3cm}

\noindent where $x_2\geq y_2$, $y_2\leq z_2$, and either $\rho_{z_2}=i_{z_2}$ or else $\rho_{z_2}=k_{z_2}$ and $z_2>y_2$.  Since $\omega_1\neq\omega_2$, there are four sub-cases: either $z_1\neq z_2$; or $z_1=z_2$ but $\sigma_{z_1}\neq \rho_{z_2}$; or $y_1\neq y_2$; or $x_1\neq x_2$.  In each of the four sub-cases below, we denote $\sigma_1=i_1$ and $\sigma_n=i_n$ if $\sigma_{z_1}=i_{z_1}$; and $\sigma_n=k_n$ and $\sigma_n=k_n$ if $\sigma_{z_1}=k_{z_1}$.  Similarly we allow $\rho_1,\rho_n$ to denote either $i_1,i_n$ or $k_1,k_n$ respectively as implied by the value of $\rho_{z_2}$.

\indent\indent \textit{Sub-Case (a):}  Suppose $z_1<z_2$.  Consider the separating space $(\mathbb{R},\tau_1,...,\tau_n)$ where $\tau_1=...=\tau_{z_1}=\tau_s$ and $\tau_{z_1+1}=...=\tau_n=\tau_u$.  Since $y_1\leq z_1$, we have we have $\omega_1=k_{x_1}f_1\sigma_1i_1f_*=k_{x_1}f_1i_1f_*$, so $\omega_1=f_1i_1f_*$ or $\omega_1=k_nf_1i_1f_*$, depending on the value of $x_1$.  On the other hand, considering all possible values of $x_2$, $y_2$, and $\rho_{z_2}=\rho_n$, we compute 

\begin{center} $\omega_2=k_{x_2}f_{y_2}\rho_ni_1f_*\in\{f_1i_nf_*,f_1k_ni_*f_*,f_ni_nf_*,k_nf_1k_ni_*f_*,k_nf_1i_nf_*\}$.
\end{center}
\vspace{.3cm}

\noindent It follows that $\omega_1A\neq\omega_2A$, where $A$ is the initial set from Example \ref{ex_usual_sorg}.

\indent\indent \textit{Sub-Case (b):}  Suppose $z_1=z_2$, but $\sigma_{z_1}=k_{z_1}$ with $z_1>y_1$, while $\rho_{z_2}=i_{z_2}$.  We take the separating space $(\mathbb{R},\tau_1,...,\tau_n)$ where $\tau_1=...=\tau_{y_1}=\tau_s$ and $\tau_{y_1+1}=...=\tau_n=\tau_u$.  We have $\omega_1=k_{x_1}f_1k_ni_1f_*\in\{f_1k_ni_*f_*,k_nf_1k_ni_*f_*\}$, while since $z_2=z_1>y_1$, we have $\omega_2=k_{x_2}f_{y_2}i_ni_1f_*=k_{x_2}f_{y_2}i_nf_*\in\{f_1i_nf_*, f_ni_nf_*,k_nf_1i_nf_*\}$.  So $\omega_1A\neq\omega_2A$, taking $A$ from Example \ref{ex_usual_sorg}.

\indent\indent \textit{Sub-Case (c):}  Suppose $y_1<y_2$, and take the separating space $(\mathbb{R},\tau_1,...,\tau_n)$ where $\tau_1=...=\tau_{y_1}=\tau_s$ and $\tau_{y_1+1}=...=\tau_n=\tau_u$.  We have

\begin{center} $\omega_1=k_{x_1}f_1\sigma_{z_1}i_1f_*\in\{f_1i_1f_*,f_1i_nf_*,f_1k_ni_*f_*,k_nf_1i_1f_*,k_nf_1i_nf_*,k_nf_1k_ni_*f_*\}$,
\end{center}
\vspace{.3cm}

\noindent whereas since $z_2\geq y_2$, we have $\omega_2=k_{x_2}f_n\rho_ni_1f_*=k_{x_2}f_ni_nf_*=f_ni_nf_*$.  So $\omega_1A\neq\omega_2A$, taking $A$ from Example \ref{ex_usual_sorg}.

\indent\indent \textit{Sub-Case (d):}  Suppose $y_1=y_2$, but $x_1<x_2$, and take the separating space $(\mathbb{R},\tau_1,...,\tau_n)$ where $\tau_1=...=\tau_{x_1}=\tau_s$ and $\tau_{x_1+1}=...=\tau_n=\tau_u$ with the initial set $A$ from Example \ref{ex_usual_sorg}.  Since $y_1=y_2\leq x_1$, we have $\omega_1=k_1f_1\sigma_{z_1}i_1f_*\in\{f_1i_1f_*,f_1i_nf_*,f_1k_ni_*f_*\}$ and $\omega_2=k_nf_1\rho_{z_2}i_1f_*\in\{k_nf_1i_1f_*, k_nf_1i_nf_*, k_nf_1k_ni_*f_*\}$, so $\omega_1A\neq\omega_2A$.\\

\noindent \underline{Case 8: $\omega_1,\omega_2\in KFIK\cup KFKIK$.}  In this case take the same separating space as in Case 7, but for an initial set take $f_nA$ where $A$ is the initial set from Case 7.  We are done if $\omega_1f_nA\neq\omega_2f_nA$; but this follows from Case 7 because $\omega_1f_n,\omega_2f_n\in KFIF\cup KFKIF$.\\

\noindent \underline{Case 9: $\omega_1,\omega_2\in KFKI\cup KFIKI$.}  Take the same separating space as in Cases 7 and 8, and for an initial set take $cA$ where $A$ is the initial set from Case 8.  Then since $\omega_1c,\omega_2c\in KFIK\cup KFKIK$, we have $\omega_1cA\neq\omega_2cA$ by Case 8. 

\section*{Acknowledgement}

We extend our sincere thanks to Mark Bowron for his comments and corrections.

\end{document}